\documentclass[12pt,a4paper]{amsart}
\usepackage{geometry}
\usepackage[utf8]{inputenc} 
\usepackage{amsmath}
\usepackage{amsthm}
\usepackage{mathabx}
\usepackage{float}
\usepackage[english]{babel} 
\usepackage[T1]{fontenc} 
\usepackage[parfill]{parskip}
\usepackage{amssymb}
\usepackage[pdftex]{graphicx}
\usepackage{epsfig}
\usepackage{epstopdf}
\usepackage{enumerate}
\usepackage{theoremref}
\usepackage{bbm}
\usepackage{indentfirst}
\usepackage{tikz}
\usepackage{tikz-cd}
\usetikzlibrary{matrix,arrows}
\usepackage[hypertexnames=false]{hyperref}

\newcommand{\dz}{\mathrm{d}z}

\newcommand{\vt}{\vartheta}

\newcommand{\C}{\mathbb{C}}

\newcommand{\Z}{\mathbb{Z}}
\newcommand{\N}{\mathbb{N}}

\newcommand{\Q}{\mathbb{Q}}

\newcommand{\cF}{\mathcal{F}}

\newcommand{\cG}{\mathcal{G}}

\newcommand{\bfs}{\bfseries}
\newcommand{\its}{\itshape}

\newcommand{\SL}{\mathrm{SL}}

\newcommand{\LRA}{\longrightarrow}
\newcommand{\LM}{\longmapsto}
\newcommand{\be}{\begin{enumerate}}
\newcommand{\ee}{\end{enumerate}}

\newcommand{\res}{\mathrm{res}}
\newcommand{\ord}{\mathrm{ord}}
\newcommand{\sgn}{\mathrm{sgn}}

\newcommand{\smal}{\left( \begin{smallmatrix}}
	\newcommand{\smar}{\end{smallmatrix} \right)}

\makeatletter
\@namedef{subjclassname@2020}{%
	\textup{2020} Mathematics Subject Classification}
\makeatother

\newtheorem{definition}{Definition}[section]

\newtheorem{bem}[definition]{Remark}

\newtheorem{lemma}[definition]{Lemma}

\newtheorem{bsp}[definition]{Example}

\newtheorem{thm}[definition]{Theorem}

\newtheorem{prop}[definition]{Proposition}

\newtheorem{cor}[definition]{Corollary}

\newenvironment{pro}{\begin{proof}[Proof]}{\end{proof}}

\newcommand{\dm}[1]{\begin{displaymath} #1 \end{displaymath}}

\newcommand{\mat}[1]{\begin{pmatrix} #1 \end{pmatrix}}

\title{L-series of Eisenstein series vanishing at critical values}
\author{Johann Franke}
\address{Mathematical Institute, University of Cologne, Weyertal 86-90, 50931 Cologne, Germany}
\email{jfrank12@uni-koeln.de}

\date{}
\begin{document}
\maketitle
\renewcommand{\abstractname}{Abstract}
\begin{abstract} 
Using the relations between rational functions and Eisenstein series, as well as the inferences for cotangent sums and period polynomials, we work out a precise description for Eisenstein series whose $L$-series vanish at certain critical values. This is possible for small weights compared to the level of the Eisenstein series. For large weights we give a partial result and determine subspaces with simultaneous vanishing properties.
\end{abstract}

\setcounter{page}{1}
\pagenumbering{arabic}

\section{Introduction}

Cotangent sums can be expressed in the form:
\begin{align*}
\sum_{n=1}^{N-1} a_n \cot\left( \frac{\pi h n}{N} \right).
\end{align*}
Here, $a_n$ are arbitrary complex numbers, and $\gcd(h,N) = 1$. They reoccur in number-theoretical contexts. For instance, by choosing $a_n := \frac{n}{N}$, the so-called {\itshape Vasyunin sum} (see \cite{BettCon})
\begin{align*}
V(h,N) := \sum_{n=1}^{N-1} \frac{n}{N} \cot\left( \frac{\pi h n}{N}\right)
\end{align*}
satisfies the reciprocity law:
\begin{align*}
V\left(h,k \right) + V\left(k,h \right) & = \frac{\log(2\pi) - \gamma}{\pi} (k+h) + \frac{k-h}{\pi} \log\left( \frac{h}{k}\right) \\
& - \frac{\sqrt{hk}}{\pi^2} \int_{-\infty}^\infty \left| \zeta\left( \frac12 + it\right) \right|^2 \left( \frac{h}{k}\right)^{it} \frac{\mathrm{d}t}{\frac14 + t^2}.
\end{align*}
Here $\zeta(s) := \sum_{n=1}^\infty n^{-s}$ denotes the {\itshape Riemann zeta function}. They have been proven to be of importance in the Nyman–Beurling criterion for the Riemann Hypothesis, see \cite{BettCon} and for more information on Nyman–Beurling, \cite{Baez}. For some applications, it is beneficial to consider generalizations
\begin{align*}
\sum_{n=1}^{N-1} a_n \cot^\nu\left( \frac{\pi h n}{N} \right)
\end{align*}
with integers $\nu \geq 0$. A classical example is
\begin{align} \label{cotIde}
\sum_{n=1}^{N-1} \cot^2 \left( \frac{\pi n}{N}\right) = \frac{(N-1)(N-2)}{3},
\end{align}
valid for all $N \in \N$. Berndt and Yeap \cite{BerYeap} were able to generalize \eqref{cotIde} to
\begin{align} \label{impressive}
\sum_{j=1}^{N-1} \cot^{2n}\left( \frac{\pi j}{N}\right) = (-1)^nN - (-1)^n 2^{2n} \sum_{j_0 = 0}^{n} \left( \sum_{\substack{j_1, ..., j_{2n} \geq 0 \\ j_0 + j_1 + \cdots + j_{2n} = n}} \prod_{r=0}^{2n} \frac{B_{2j_r}}{(2j_r)!}\right) N^{2j_0}
\end{align}
valid for $n, N \in \N$. Here, $B_n$ denotes the $n$-th Bernoulli number. The proofs behind these identities are of complex analytic nature. In \cite{Franke3}, \eqref{impressive} was extended to the case of Dirichlet characters, establishing a connection to values of Dirichlet $L$-functions $L(\chi;s) := \sum_{n=1}^\infty \chi(n)n^{-s}$. In this paper, we take a step further and provide a structural description of the mentioned connections. These structures are ``modular'' in nature. The background for this is the interplay between rational functions and modular forms, which was first elaborated in \cite{Franke1}. Note that the relationship between so called hyperbolic cotangent sums and Eisenstein series has also already been considered by Berndt and Straub \cite{BerStraub} in the context of Ramanujan identities for odd zeta values. A crucial ingredient for our point of view is the concept of weak functions, which are meromorphic and 1-periodic functions defined throughout $\C$, having simple poles only at rational places and vanishing at $\pm i\infty$. The location of poles at $\frac{1}{N}\mathbb{Z}$ determines the level $N$ of a weak function $\omega \in W_N$. Pairs of weak functions of possibly different levels can generate modular forms, thus giving rise to a linear map from $W_M \otimes W_N$ to a space of modular forms. These show up to be Eisenstein series with specific properties. In this paper, we exploit the Eichler duality developed in \cite{Franke3} to provide precise statements about values of $L$-series for certain Eisenstein series. At this point, we recall that to any modular form $f(\tau) = \sum_{n=0}^\infty a_n e^{\frac{2\pi i n}{\lambda}}$ of weight $k$ for a congruence subgroup $\Gamma$ (where $\lambda > 0$ is some contant determined by $\Gamma$) we define the corresponding $L$-series by
\begin{align} \label{eq:L-Def}
L(f;s) := \sum_{n=1}^\infty \frac{a_n}{n^s} \qquad (\mathrm{Re}(s) > k).
\end{align}
Let $k \geq 3$ be an integer. The spaces $\mathcal{E}_k(\Gamma_1(p_1p_2))_0^{(p_1,p_2)}$ prove advantageous, spanning exactly from
\begin{align*}
E_k\left( \chi, \psi; p_2\tau\right) := \sum_{(m,n) \in \Z^2 \setminus {(0,0)}} \chi(m)\psi(n)(mp_2\tau + n)^{-k}
\end{align*}
with odd prime numbers $p_1, p_2$ and nontrivial Dirichlet characters $\chi$ and $\psi$ modulo $p_1$ and $p_2$ such that $\chi(-1)\psi(-1) = (-1)^k$. 

\begin{thm}[see also Theorem \ref{thm:mainsmallweight}] \label{thm:A} Let $p_1, p_2$ be odd primes, $3 \leq k \leq \min\{p_1-2,p_2-2\}$ an integer, and $0 \leq \ell_1 \leq p_1-2$, $0 \leq \ell_2 \leq p_2-2$, with $\ell_1 + \ell_2 \leq k-1$. Let $\mathcal{S} \subset \{0, \ldots, k-2\}$ be an arbitrary subset. Then there exists a computable subspace $U \subset W_{p_1}^0 \otimes W_{p_2}^0$ (depending on $p_1$, $p_2$ and $\mathcal{S}$) and an equally computable linear map $\xi \colon U \to \mathcal{E}_k(\Gamma_1(p_1p_2))_0^{ (p_1, p_2)}$, such that the following sequence is exact:
	\begin{align*}
	0 \longrightarrow U \overset{\xi}{\longrightarrow} \mathcal{E}_k(\Gamma_1(p_1p_2))_0^{ (p_1, p_2)} \overset{\mathcal{L}_{\mathcal{S},k}}{\longrightarrow }\C^{|\mathcal{S}|} \longrightarrow 0,
	\end{align*}
	where the linear map $\mathcal{L}_{\mathcal{S},k} \colon \mathcal{E}_k(\Gamma_1(p_1p_2))_0^{ (p_1, p_2)} \to \C^{|\mathcal{S}|}$ is given by
	\begin{align*}
	\mathcal{L}_{\mathcal{S},k}(f) :=  \frac{(-2\pi i)^k}{(k-2)!} \frac{p_2^{k-1}}{4\pi^2} \left( \binom{k-2}{\ell} i^{1-\ell} (2\pi)^{-\ell-1} \Gamma(\ell+1) L\left( f; \ell+1\right) p_1^\ell \right)_{\ell \in \mathcal{S}}.
	\end{align*}
\end{thm}
We would like to emphasize at this point that the shape of the spaces $U$ results precisely from the formalism of the weak functions. Moreover, how to precisely compute the space $U$ and the map $\xi$ is described further below. It should be emphasized that this is also achievable algorithmically by determining the kernel of a Vandermonde matrix with cotangent bases, see Proposition \ref{ordform}. Background of this is the interplay between cotangent sums and Taylor coefficients of weak functions, as obtained in \cite{Franke3}. This process gives rise to the selection of an appropriate basis for $W_{p_1}^0 \otimes W_{p_2}^0$ to describe $U$. With Theorem \ref{thm:A}, vector spaces of modular forms, whose $L$-values vanish at specific critical points, can be precisely described, provided that the weight $k$ is not too large compared to the primes $p_1$ and $p_2$. We can also derive dimension formulas in these cases, see Corollary \ref{cor:dimension}.  
\par
\,\, The case of a large weight $k$ is more difficult, and precise statements are probably no longer possible here with our method. However, one can give estimates in the form that modular forms can be calculated whose $L$-values vanish in certain ranges. Of particular interest here are ``boundary regions'' of the form $\{1, \ldots, \ell_1\} \cup \{k-\ell_2, \ldots, k-1\}$, because the simultaneous vanishing of $L$-values can be directly related here to the complex analytic zero orders of weak functions. The definition of these orders are provided in Definition \ref{def:ordercomplex}.

\begin{thm}[see also Theorem \ref{thm:L-strips}] \label{thm:B} Let $p_1$ and $p_2$ be two odd prime numbers and $k \geq 3$ be an integer.  Let $\ell_1$ and $\ell_2$ be integers such that $\max\{0, p_2-k-1\} \leq \ell_1 \leq p_1-2$, $\max\{0, p_1-k-1\} \leq \ell_2 \leq p_2-2$, and $\ell_1+\ell_2 \leq k-1$. We assign $\ell_1$ and $\ell_2$ a space 
	\begin{align*}
	\mathcal{E}_k^{(\ell_1, \ell_2)}(\Gamma_1(p_1p_2))_0^{(p_1,p_2)} := \left< f \in \mathcal{E}_k(\Gamma_1(p_1p_2))_0^{(p_1,p_2)}  \Big{|} \ \ord\left(\widetilde{\vartheta}_k^{-1}(f)\right) \geq (\ell_1, \ell_2) \right>,
	\end{align*}
	where $\widetilde{\vartheta}_k$ is the isomorphism defined in \eqref{eq:tildetheta}. Then we have
	\begin{align*}
	\mathcal{E}_k^{(\ell_1, \ell_2)}(\Gamma_1(p_1p_2))_0^{(p_1,p_2)} \subset \mathcal{E}_k^{\{0,\ldots,\ell_1-1\} \cup \{k-1-\ell_2,\ldots,k-2\}}(\Gamma_1(p_1p_2))_0^{(p_1,p_2)},
	\end{align*}
	where $\mathcal{E}_k^{\{0,\ldots,\ell_1-1\} \cup \{k-1-\ell_2,\ldots,k-2\}}(\Gamma_1(p_1p_2))_0^{(p_1,p_2)}$ is the subspace of all elements \\
	$f \in \mathcal{E}_k(\Gamma_1(p_1p_2))_0^{(p_1,p_2)}$ such that $L(f;j) = 0$ for all $1 \leq j \leq \ell_1$ and $k-\ell_2 \leq j \leq k-1$. 
\end{thm}

Note that Theorem \ref{thm:B}, quite unlike Theorem \ref{thm:A}, is particularly applicable to large weights. As far as the author knows, cusp forms do not come from rational functions, so the methods presented do not seem to be applicable to these types of modular forms. Nevertheless, critical $L$-values in the context of cusp forms are of utmost importance for mathematics, for example in the context of the conjecture of Birch and Swinnerton-Dyer (see for example \cite{Huse}, \cite{Knapp}). Note that recently Males, Mono, Rolen and Wagner \cite{MMRW} characterized the vanishing of twisted central $L$-values attached to newforms of square-free level in terms of so-called local polynomials and the action of finitely many Hecke operators thereon. \\
\begin{bsp} We demonstrate the numerical aspect of Theorem \ref{thm:B} by providing an explicit example. Let $\chi_5$ be the Dirichlet character modulo $5$ with $\chi_5(2) = i$. Note that the interested reader can find more details in Example \ref{bsp:55-bsp}. For every even weight $k \geq 4$, the $\Gamma_1(25)$ modular form 
		\begin{align*}
	f(\tau) = C E_k(\overline{\chi_5}, \overline{\chi_5}; 5\tau) + 5 \left( E_k(\overline{\chi_5},\chi_5; 5\tau) +  E_k(\chi_5, \overline{\chi_5}; 5\tau)\right) + \overline{C} E_k(\chi_5, \chi_5; 5\tau),
	\end{align*}
	where
		\begin{align*}
	C := - i (-3-4 i)^{\frac{3}{4}} (-3+4 i)^{\frac14},
	\end{align*}
has the $L$-function
\begin{align*}
L(f;s) & = \frac{2(-2\pi i)^k}{5^k (k-1)!} \left( C \mathcal{G}(\overline{\chi_5}) L(\overline{\chi_5};s)L(\chi_5;s-k+1) + 5 \mathcal{G}(\chi_5)  L(\overline{\chi_5};s)L(\overline{\chi_5};s-k+1)\right. \\
& \quad \left. +  5 \mathcal{G}(\overline{\chi_5}) L(\chi_5;s)L(\chi_5;s-k+1) + \overline{C} \mathcal{G}(\chi_5) L(\chi_5;s)L(\overline{\chi_5};s-k+1) \right).
\end{align*}
	It satisfies
	\begin{align*}
	L(f;1) = L(f;2) = L(f;k-2) = L(f;k-1) = 0.
	\end{align*}
	Note that, as $\chi_5$ is an odd character, only the zeros at $s = 1$ and $s = k-1$ are non-trivial. 
	\end{bsp}
The paper is organized as follows. In Section 2 we recall some preliminaries in the theory of Eisenstein series and weak functions. In Section 3 we develop a theory of orders of weak functions, and establish connections to the ordering of holomorphic functions of several variables. Finally, in Section 4 we apply this to certain Eisenstein series and their $L$-series.

\section*{Acknowledgments}

The author is very grateful to Kathrin Bringmann, Nikolaos Diamantis and Andreas Mono for numerous comments regarding an earlier version of the paper. 

\section{Preliminaries} 
In this section, we collect some basic facts and known results needed for our later investigations.  

\subsection{Discrete Fourier transforms} For any Dirichlet character $\psi$ modulo $N$ we define the Gauss sum $\cG(\psi) := \sum_{n=0}^{N-1} \psi(n)e^{2\pi i n/N}$. For the generalized Gauss sum it will be more convenient to use the more general notion of the {\it discrete Fourier transform}, which is defined for $N$-periodic functions $f$ by
\dm{(\cF_N f)(j) := \sum_{n=0}^{N-1} f(n)e^{-\frac{2\pi i j n}{N}}.}
Note that we have an inverse transformation 
\dm{(\cF^{-1}_N g)(j) := \frac1N \sum_{n=0}^{N-1} g(n)e^{\frac{2\pi i j n}{N}}.} 

\subsection{Eisenstein series}

In this section, we recall some basic facts about Eisenstein series. 

\begin{definition} For two Dirichlet characters $\chi$ and $\psi$ modulo $N_\chi$ and $N_\psi$ and any integer $k \geq 3$ we define the corresponding Eisenstein series via 
	\begin{align}
	\label{EisensteinDefinition} E_k\left( \chi, \psi; \tau\right) := \sum_{(m,n) \in \Z^2 \setminus \{(0,0)\}} \chi(m)\psi(n)(m\tau + n)^{-k}.
	\end{align} 
\end{definition}

As usual, we denote
\begin{align*}
\Gamma_1(N) := \left\{ \mat{a & b \\ c & d} \in \SL_2(\Z) \Big{|} \mat{a & b \\ c & d} \equiv \mat{1 & * \\ 0 & 1} \ \pmod{N} \right\}.
\end{align*}
The following is well-known.
\begin{thm}[see Chapter 7 of \cite{Miy}] \label{T:Eisenstein}Let $k \in \N$ and $\chi$ and $\psi$ be Dirichlet characters modulo $N_\chi$ and $N_\psi$ satisfying $\chi(-1)\psi(-1) = (-1)^k$. Then we have the following. 
	\begin{enumerate}[(i)]
		\item Every Eisenstein series admits a Fourier series that converges on the entire upper half plane. It is given by
		\begin{align*} 
		\hspace{0.8cm} E_k\left( \chi, \psi; \tau\right) = 2L(\psi; k)\chi(0)+ \frac{2(-2\pi i)^k}{N_\psi^k (k-1)!}\sum_{m=1}^\infty \left( \sum_{d|m} d^{k-1} (\cF_{N_\psi}\psi)(-d) \chi \left( \frac{m}{d}\right) \right) q^{\frac{m}{N_\psi}},
		\end{align*}
		where as usual $q := e^{2\pi i \tau}$. In particular, if $\psi$ is a primitive character, this simplifies to 
		\begin{align*} 
		E_k\left( \chi, \psi; \tau\right) = 2L(\psi; k)\chi(0)+ \frac{2(-2\pi i)^k}{N_\psi^k (k-1)!}\sum_{m=1}^\infty \left( \sum_{d|m} d^{k-1} \overline{\psi}(d) \chi \left( \frac{m}{d}\right) \right) q^{\frac{m}{N_\psi}}.
		\end{align*}
		\item All the $E_k(\chi, \psi; \tau)$ are holomorphic modular forms of weight $k$ for the group 
		\begin{align*}
		\Gamma_0(N_\chi, N_\psi) := \left\{ \begin{pmatrix} a & b \\ c & d\end{pmatrix} \in \mathrm{SL}_2(\Z) \Big{|} c \equiv 0 \ \pmod{N_\chi}, b \equiv 0 \ \pmod{N_\psi}\right\}.
		\end{align*}
		More precisely, one has the transformation law 
		\begin{align*}
		E_k(\chi, \psi; \tau)|_k \gamma = \nu(\gamma) E_k(\chi, \psi; \tau),
		\end{align*}
		where the multiplier system $\nu : \Gamma_0(N_\chi, N_\psi) \to \C^\times$ is defined by $\nu \left(\begin{smallmatrix} a & b \\ c & d \end{smallmatrix}\right) := \chi(d) \overline{\psi}(d)$. In particular, for integers $N_\chi \mid M$ and $N_\psi \mid N$ all $E_k(\chi, \psi; N\tau)$ are modular forms of weight $k$ for the congruence subgroup $\Gamma_1(MN)$ with trivial multiplier system. 
	\end{enumerate}
\end{thm}

In general, Eisenstein series can be seen as the counter part of cusp forms in the theory of modular forms. Let $\mathcal{M}_k(\Gamma)$ be the space of weight $k$ modular forms for the congruence subgroup $\Gamma$. It is easy to see that the subset of cusp forms, i.e., modular forms that vanish in all cusps $\Q \cup \{i\infty\}$, forms a subspace $\mathcal{S}_k(\Gamma)$ of $\mathcal{M}_k(\Gamma)$. The so called Eisenstein space can be defined as the quotient 
\begin{align*}
\mathcal{E}_k(\Gamma) := \mathcal{M}_k(\Gamma) / \mathcal{S}_k(\Gamma).
\end{align*}

In the following, we mainly focus on the congruence subgroups $\Gamma_1(N)$ since it appears that this choice is the most natural for our purposes. The useful proposition below presents a basis for the space $\mathcal{E}_k(\Gamma_1(N))$. 

\begin{thm}[see Theorem 4.5.2 of \cite{Diam}] \label{T:Basis} Let $k \geq 3$ and $N > 2$ be integers. Let the set $A_{N,k}$ consist of all triples $(\chi, \psi, t)$, where $\chi$ and $\psi$ are characters modulo $N_\chi$ and $N_\psi$, respectively, and $t$ is a positive integer, such that the following is satisfied:
\begin{enumerate}[(i)]
		\item The characters $\psi$ and $\chi$ are primitive and satisfy $\chi(-1)\psi(-1) = (-1)^k$. 
		\item For the moduli $N_\chi$ and $N_\psi$, the divisor relation $N_\chi N_\psi t \mid N$ holds. 
\end{enumerate}
Then, the system $\left\{E_k\left(\chi, \psi; t N_\psi \tau \right)\right\}_{(\chi, \psi, t) \in A_{N,k}}$ defines a basis of $\mathcal{E}_k(\Gamma_1(N))$.
	\end{thm}

Let $f(\tau) = \sum_{n=0}^\infty a_n q^{\frac{n}{\lambda}}$ be a modular form of weight $k$ for a congruence subgroup $\Gamma$ with $L$-series $L(f;s) = \sum_{n=1}^\infty a_n n^{-s}$. Note that the parameter $\lambda > 0$ is chosen to be the width of the cusp $i\infty$ with respect to $\Gamma$. It makes sense considering the completed $L$-series, usually denoted by $\Lambda$:
\begin{align*}
\Lambda(f;s) := \left( \frac{2\pi}{\lambda}\right)^{-s} \Gamma(s) L(f;s) = \int_0^\infty f(ix) x^{s-1} \mathrm{d}x.
\end{align*}
As we have $a_n = O(n^{k-1})$ for modular forms of weight $k \geq 3$ the series for $L(f;s)$ converges absolutely for values $s$ with $\mathrm{Re}(s) > k$. It can be continued to a holomorphic function on $\C \setminus \{k\}$ with a possible pole at $s=k$ and satisfies a functional equation. One can describe the $L$-series corresponding to Eisenstein series in terms of Dirichlet $L$-series as follows. 

\begin{thm}[see Theorem 4.7.1. and p. 271 of \cite{Miy}] \label{thm:L-functions} Let $\chi$ and $\psi$ be primitive Dirichlet characters modulo $N_\chi$ and $N_\psi$, respectively. Then we obtain for $f(\tau) := E_k(\chi, \psi; N_\psi \tau)$ the $L$-series
	\begin{align*}
	L(f; s) = \frac{2(-2\pi i)^k\mathcal{G}(\psi)}{N_\psi^k (k-1)!} L(\chi; s) L\left(\overline{\psi}; s - k + 1\right).
	\end{align*}
	\end{thm}

As in the case of cusp forms, one can also define the concept of newforms for Eisenstein series. 

\begin{definition} \label{def:newforms} Let $N$ be a positive integer and $\chi$, $\psi$ be primitive characters modulo $N_\chi$ and $N_\psi$, respectively, such that $N = N_\chi N_\psi$ and $\chi(-1)\psi(-1) = (-1)^k$. Then we call the Eisenstein series $E_k(\chi, \psi; N_\psi \tau)$ a newform of level $N$. We denote the space generated by newforms of level $N$ by $\mathcal{E}_k(\Gamma_1(N))^{\mathrm{new}}$. In the case of fixed conductors $u$ and $v$ with $uv = N$, we write
	\begin{align*}
	\mathcal{E}_k(\Gamma_1(N))_0^{\mathrm{new}, (u,v)} := \left< E_k(\chi, \psi; v\tau) \Big{|} \chi \in \mathcal{C}^{\mathrm{prim}}_0(u), \psi \in \mathcal{C}^{\mathrm{prim}}_0(v)\right>,
	\end{align*}
	which is clearly a subspace of $\mathcal{E}_k(\Gamma_1(N))^{\mathrm{new}}$. Here, $\mathcal{C}^{\mathrm{prim}}_0(N)$ is the set of all non-principal, primitive characters modulo $N$. 
	\end{definition}

In our following work, it makes sense that we specialize in modular forms that vanish at the cusps $0$ and $i\infty$. In the case of Eisenstein series, this means that we have to restrict ourselves to non-principal characters. This is an easy consequence of Theorem \ref{T:Eisenstein} (i). This motivates our definition of the spaces  $\mathcal{E}_k(\Gamma_1(N))_0^{\mathrm{new}, (u,v)}$ above. The reader is reminded that we will assign a $0$ to vector spaces that are constrained with respect to this principle. For example, on the weak functions side, this means that we have a removable singularity in $z=0$. In this work, we focus on spaces $\mathcal{E}_k(\Gamma_1(MN))$, where $M, N \in \N$. It is natural for us to only consider the newforms with characters modulo $M$ and $N$, respectively, and we denote the corresponding subspace by 
\begin{align} 
\nonumber &\hspace{1.5cm} \mathcal{E}_k(\Gamma_1(MN))_0^{(M, N)} \\
\label{EDef}  & \hspace{1cm} := \left< E_k(\chi, \psi; N\tau) \Big{|} \chi \in \mathcal{C}_0^{\mathrm{prim}}(N_\chi), \psi \in \mathcal{C}_0^{\mathrm{prim}}(N_\psi), N_\chi \mid M, N_\psi \mid N, (\chi \psi)(-1) = (-1)^k \right>.
\end{align}
It is easy to verify, that this is indeed a subspace of $\mathcal{E}_k(\Gamma_1(MN))$ and that the generating elements are linearly independent.
\begin{prop} \label{prop:linearindep} We have $\mathcal{E}_k(\Gamma_1(MN))_0^{(M, N)} \subset \mathcal{E}_k(\Gamma_1(MN))$, and the generating elements in \eqref{EDef} are linearly independent. 
 \end{prop}

\begin{proof} By assumption, the characters $\chi$ modulo $N_\chi$ and $\psi$ modulo $N_\psi$ are primitive. Since we have 
	\begin{align*}
	E_k\left( \chi, \psi; N\tau \right) = E_k\left( \chi, \psi; \frac{N}{N_\psi} N_\psi\tau \right)
	\end{align*}
	with $t = \frac{N}{N_\psi}$ and $N_\chi N_\psi t = N_\chi N \mid MN$, the claim follows immediately with Theorem \ref{T:Basis}. 
	\end{proof}

\begin{bem}Note that, by Theorem \ref{T:Eisenstein}, all modular forms in space $\mathcal{E}_k(\Gamma_1(MN))_0^{(M, N)}$ have the property of vanishing in the cusps $0$ and $i\infty$. 
	\end{bem}

We set $B_d$ to be the linear operator 
\begin{align*}
(f|B_d)(z) := f|_k\left( \begin{smallmatrix} d & 0 \\ 0 & 1\end{smallmatrix}\right)(z) = f(dz),
\end{align*}
for arbitrary values of $k$. Shifting the argument of a modular form in the above way does essentially not effect its completed $L$-series, since we have the formula
\begin{align} \label{eq:Lambda-Trans}
\Lambda(f|B_d; s) = d^{-s} \Lambda(f; s).
\end{align}

\subsection{Weak functions and modular forms}

In this section we recall the theory around weak functions, that were introduced as rational functions in \cite{Franke1} in the context of modular forms. 
\begin{definition} Let $N$ be a positive integer. We call a meromorphic function on the entire plane $\omega$ a weak function of level $N$, if it is 1-periodic, holomorphic in $\C \setminus \tfrac{1}{N}\Z$ with possible poles of order at most 1 in $z = \tfrac{j}{N} \in \Q$ and satisfies the grwoth condition 
\begin{align*}
\omega(x + iy) = O\left(|y|^{-A}\right), \qquad y \rightarrow \infty,
\end{align*}
for all values $A > 0$. We collect all weak functions with level $N$ in the vector space $W_N$. 
\end{definition}
\begin{bem} As in the theory of modular forms, the notion of level is not uniquely determined at first, since this can be raised by lifts to smaller congruence subsets. This is differentiated with the concept of new and old forms. Something similar is possible on the weak functions side. So we can speak of a new function with level $N$ if there is a real pole at a point $\frac{j}{N}$ with $\gcd(j,N) = 1$. 
	\end{bem}
By Liouville's theorem it is immediate  that 
\begin{align*}
\omega(z) = \sum_{j=1}^N \beta^\omega(j) \frac{e(z)}{e(\frac{j}{N}) - e(z)}, \qquad e(z) := e^{2\pi i z},
\end{align*}
for some coefficients $\beta^\omega(j)$ satsifying 
\begin{align} \label{eq:betasum}
\beta^\omega(1) + \beta^\omega(2) + \cdots + \beta^\omega(N) = 0.
\end{align}
 In particular, the spaces $W_N$ are finite dimensional. In addition, weak functions are very closely linked to Eisenstein series. To see this, for an integer $k$ and a pair $\omega \otimes \eta$ define the following holomorphic function on the upper half plane:
\begin{align} \label{vartheta}
\vartheta_k(\omega \otimes \eta; \tau) := -2\pi i \sum_{x \in \Q^\times} \res_{z=x}\left( z^{k-1} \eta(z) \omega(\tau z)\right).
\end{align}

Due to symmetry arguments, the map $\vartheta_k$ on $W_M \otimes W_N$ is highly non-injective. Depending on $k$, it makes sense to restrict to suitable subspaces. Let $W_N = W_N^+ \oplus W_N^-$ be the decomposition into even and odd functions, respectively. Then we put
\begin{align*}
(W_M \otimes W_N)_k := \begin{cases} W_M^+ \otimes W_N^+ \oplus W_M^- \otimes W_N^-, & \qquad \text{if } k \text{ is even}, \\ W_M^+ \otimes W_N^- \oplus W_M^- \otimes W_N^+, & \qquad \text{if } k \text{ is odd}.\end{cases}
\end{align*}

We also say that even and odd functions $\omega$ have positive or negative sign $\mathrm{sgn}(\omega)$, respectively, and put (if possible) $\sgn(\omega \otimes \eta) := \sgn(\omega)\sgn(\eta)$. As $\vartheta_k(\omega \otimes \eta; \tau) $ induces a periodic function, so we can find a Fourier series. 

\begin{prop} For $\omega \otimes \eta \in (W_M \otimes W_N)_k$ and $k \geq 3$, we have the Fourier expansion 
	\begin{align*}
\vt_k(\omega \otimes \eta; \tau) = 2N^{1-k}\sum_{m=1}^\infty \sum_{d|m} \left( d^{k-1} \beta^\eta(d) \left( \cF_M  \beta^\omega \right) \left( \frac{m}{d} \right)\right) q^{\frac{m}{N}}.
	\end{align*}
	\end{prop}
This function satisfies an important transformation law. 
\begin{thm} \label{Transform} Let $\omega$ and $\eta$ be weak functions of level $N$. We then have $\vt_{k}(\omega \otimes \eta; \tau + N) = \vt_{k}(\omega \otimes \eta; \tau)$ and
	\begin{align} \label{transform} \vt_k\left( \omega \otimes \eta; - \frac{1}{\tau}\right) = - \tau^{k} \vt_k(\eta \otimes \widehat{\omega}; \tau) + 2\pi i \, \res_{z=0}\left( z^{k-1}\eta(z) \widehat{\omega}\left( \frac{z}{\tau}\right)\right),
	\end{align}
	where $\widehat{\omega}(z) := \omega(-z)$ is again weak of level $N$. 
\end{thm}

The concepts of Fourier transforms are useful when expressing terms of the form $\vt_k(\omega \otimes \eta; \tau)$ as modular forms, as the followig proposition shows. 
\begin{prop}[see \cite{Franke2}] Let $\chi$ and $\psi$ be non-principal Dirichlet characters modulo $M$ and $N$, respectively. Then the following identity holds: 
\begin{align} \label{eisentheta}
E_k(\chi, \psi; \tau) = \frac{\psi(-1) (-2\pi i)^{k}}{N (k-1)!} \vt_k\left(\omega_{\cF_M^{-1}(\chi)} \otimes \omega_{\cF_N(\psi)}; \tau\right).
\end{align}
\end{prop}
In particular, if $\chi$ and $\psi$ are primitive and hence conjugate up to a constant under the Forurier transform, formula  \eqref{eisentheta} simplifies to the important identity 
\begin{align}\label{eq:weakeisenstein}
E_k(\chi, \psi; \tau) = \frac{\chi(-1) (-2\pi i)^{k} \cG(\psi)}{N (k-1)! \cG(\overline{\chi})}\vt_k(\omega_{\overline{\chi}} \otimes \omega_{\overline{\psi}}; \tau).
\end{align}

In summary, the following statements can be made about purely complex analytic means:

\begin{thm} \label{A} Let $k \geq 3$ and $M, N > 1$ be integers and define the congruence subgroup 
	\dm{\Gamma_1(M, N) := \left\{\mat{a & b \\ c & d} \in \Gamma_0(M, N) \Big{|} a \equiv d \equiv 1 \pmod{M N} \right\}.} 
	Let $M_k(\Gamma_1(M, N))$ the space of weight $k$ holomorphic modular forms for $\Gamma_1(M, N)$. There is a homomorphism 
	\dm{W_{M} \otimes W_{N} \LRA M_k(\Gamma_1(M, N))}
	\dm{\omega \otimes \eta \LM \vt_{k}(\omega \otimes \eta; \tau) := -2\pi i\sum_{x \in \Q^\times} \res_{z=x}\left(z^{k-1} \eta(z) \omega(z\tau)\right).}
	In the case that $k = 1$ and $k=2$ the map stays well-defined under the restriction that the function $z \mapsto z^{k-1} \eta(z)\omega(z\tau)$ has a removable singularity in $z=0$. 
\end{thm}
In some situations, the occurrence of singularities in $z=0$ is pathological. For instance, Theorem \ref{A} shows that in these cases modular forms are no longer generated for small weights. In this paper we also want to avoid such singularities. Therefore, it is convenient that we restrict ourselves to the subspaces $W_N^0 \subset W_N$ consisting of all weak functions which have removable singularities in $z=0$. Of course we can (and need) to play the same game with $W_N^0$ regarding even and odd parts, and put
\begin{align*}
(W_M^0 \otimes W_N^0)_k := (W_M \otimes W_N)_k \cap (W_M^0 \otimes W_N^0).
\end{align*}

In addition, it makes sense to include Dirichlet character theory when choosing a basis of $W_N^0$. This has the great advantage that we get correspondences to Eisenstein series on characters on the other side. The situation of a prime level is particularly easy. 

\begin{prop} \label{prop:CharacterBasis} Let $p$ be an odd prime. For any Dirichlet character $\chi$ modulo $p$, put
	\begin{align*}
	\omega_\chi(z) := \sum_{j=1}^{p-1} \chi(j) \frac{e(z)}{e(\frac{j}{p}) - e(z)}.
	\end{align*} 
	Then, $(\omega_\chi)_{\chi \in \mathcal{C}^{\mathrm{prim}}_0(p)}$ is a basis of $W_p^0$.  Furthermore, the sets $(\omega_{\chi \in \mathcal{C}^{\mathrm{prim}}_0(p), \chi(-1) = \mp 1})$ define a basis for $W_p^{0,\pm}$, respectively. 
	\end{prop}
\begin{proof} Every weak function $\omega \in W_p^0$ has a removable singularity in $z=0$, so it is of the form
	\begin{align*}
	\omega(z) = \sum_{j=1}^{p-1} \beta^\omega(j) \frac{e(z)}{e(\frac{j}{p}) - e(z)}.
	\end{align*}
	The assertion now follows with the facts that every nonprincipal Dirichlet character modulo the prime $p$ is already primitive, the orthogonality relations and \eqref{eq:betasum}, and that these characters are all linearly independent. For the second claim note that $\omega_\chi$ is an even or odd function, if and only if the generating coefficients $\beta^\omega$ define odd or even functions, respectively. 
	\end{proof}
Of course, similar statements apply to non-prime levels as well, but in this paper we want to focus on prime levels, so we will not go into this further. 

As in the case of primitive characters $\chi$ and $\psi$, the functions $\vt_k(\omega_\chi \otimes \omega_\psi; \tau)$ are essentially the corresponding Eisenstein series,  we can easily conclude the following proposition.

\begin{prop} \label{prop:prime-injective} Let $p_1$ and $p_2$ be prime numbers and $k \geq 3$ an integer. Then the linear map 
	\begin{align*} \vartheta_k \colon \left( W_{p_1}^0 \otimes W_{p_2}^0\right)_k & \longrightarrow \mathcal{E}_k(\Gamma_1(p_1 p_2))_0^{(p_1, p_2)}  \\
	\omega \otimes \eta & \longmapsto \vt_k(\omega \otimes \eta; p_2 \tau)
	\end{align*}
	is an isomorphism.
	\end{prop}
\begin{proof} The elements $(\omega_{\chi_1} \otimes \omega_{\chi_2})_{\chi_1 \in \mathcal{C}_0^{\mathrm{prim}(p_1)}, \chi_2 \in \mathcal{C}_0^{\mathrm{prrim}}(p_2), \chi_1(-1)\chi_2(-1) = (-1)^k}$ define a basis of $(W^0_{p_1} \otimes W^0_{p_2})_k$ by Proposition \ref{prop:CharacterBasis}. Recall that with \eqref{eq:weakeisenstein}
	\begin{align*}
	E_k(\chi, \psi; p_2\tau) = \frac{\chi(-1) (-2\pi i)^{k} \cG(\psi)}{N (k-1)! \cG(\overline{\chi})}\vt_k(\omega_{\overline{\chi}} \otimes \omega_{\overline{\psi}}; p_2\tau),
	\end{align*}
	so there is a 1-1-correspondence between basis vectors of $\left( W_{p_1}^0 \otimes W_{p_2}^0\right)_k $ and $\mathcal{E}_k(\Gamma_1(p_1 p_2))_0^{(p_1, p_2)}$ by Proposition \ref{prop:linearindep} (note that the only divisors of $p_1$ and $p_2$ are $\{1, p_1\}$ and $\{1,p_2\}$, but there are no non-principal characters modulo $1$). The claim now follows. 
	\end{proof}
Another tool which we require is the \textit{Fourier transform of a weak function}. We transform the coefficient function $\beta^\omega(j)$ of some weak function $\omega$ and use the result to construct a new weak funtion. If $\beta(N) = 0$, we obtain 
\begin{align*}
\sum_{n=1}^N \mathcal{F}_N(\beta)(n) = \sum_{n=1}^N \sum_{j=1}^N \beta(j)e^{-\frac{2\pi i n j}{N}} = N\beta(N) = 0.
\end{align*}
On the other hand, we clearly have $\cF_N(\beta)(0) = 0$ and hence the function $\cF_N(\beta)$ defines again a weak function
\begin{align*}
\cF_N \omega(z) := \sum_{n=1}^{N-1} \cF_N(\beta)(n) \frac{e(z)}{e\left( \frac{n}{N}\right) - e(z)}.
\end{align*}
This gives rise to the statement that $\cF_N$ defines an automorphism on the space $W_N^{\mathrm{ord} \geq 0}$, when considering the inverse transform 
\begin{align*}
\cF_N^{-1}(\beta)(j) := \frac{1}{N} \sum_{n=1}^N \beta(n) e^{\frac{2\pi i j n}{N}}.
\end{align*}

\subsection{Eichler integrals and period polynomials}

To any modular form $f(\tau) = \sum_{n \geq 0} a_n q^{\frac{n}{\lambda}}$ of weight $k \geq 2$ for some congruence subgroup that vanishes in the cusps in $\tau = 0$ and $\tau = i\infty$, we can associate an Eichler integral. It has the form
\dm{\mathcal{I}(f; \tau) := \frac{(-2\pi i)^{k-1}}{(k-2)!} \int \limits_\tau^{i\infty} f(z)(z - \tau)^{k-2} \dz.}
This integral represents a holomorphic and periodic function on the upper half plane and is tied to the so-called {\its period polynomial} $P(f; \tau)$ of $f$ via the functional equation 
\dm{\mathcal{I}\left(f; \tau \right) - (-1)^k \tau^{k-2} \mathcal{I}\left(f^*; -\frac{1}{\tau}\right) =: P(f; \tau),}
where $f^* = f|_k\smal 0 & -1 \\ 1 & 0\smar$. Explicitely, we have a correspondence to the critical values of the $L$-series associated to $f$ via
\dm{P(f; \tau) = (-1)^k \sum_{n=0}^{k-2} {k-2 \choose n} i^{1-n} \Lambda(f; n+1) \tau^{k-2-n}.}

We mention that period polynomials have several applications. They arise natuarally in the context of the Eichler-Shimura isomorphism (see \cite{Cohen}, Chapter 11), derivatives of $L$-functions \cite{DiamRol}, Manin's Periods Theorem \cite{Man} and the theory of transcendental numbers \cite{Gun}.

Since Eisenstein series come from rational functions, we can express their Eichler integrals in terms of residues of weak functions. The easiest case of primitive characters is presented in the following theorem. 

\begin{thm}[see Theorem 4.15 of \cite{Franke3}] \label{MainPoly} Let $k \geq 3$ be an integer, $\chi$ and $\psi$ be two primitive Dirichlet characters with $\chi(-1)\psi(-1) = (-1)^k$ and $f(\tau) = E_k(\chi, \psi; \tau)$. We then have the following identity between rational functions: 
	\dm{\sum_{\ell = 0}^{k-2} {k - 2 \choose \ell} i^{1-\ell} \Lambda(f; \ell + 1) \tau^{\ell} = \frac{4\pi^2 \chi(-1)}{N_\psi^{k-1} N_\chi (k-1)}  \res_{z=0}\left( z^{1-k} \omega_\psi(z) \omega_{\chi} \left( \frac{N_\psi z \tau}{N_\chi} \right)\right) .} 
\end{thm}

This is a special consequence of a duality principle called  {\it Eichler duality}, which was formulated using rational functions in \cite{Franke3}. It takes on a particularly simple form here, as the characters involved are primitive and therefore (except for complex conjugation) eigenfunctions under the discrete Fourier transformation. Note that we can reformulate this theorem regarding the choice of the modular form $f$. When setting $g(\tau) := (f|B_{N_\psi})(\tau)  = E_{k}(\chi, \psi; N_\psi \tau)$ in $\mathcal{E}_k(\Gamma_1(N_\psi N_\chi))$, one finds with \eqref{eq:Lambda-Trans} that 
\begin{align} \label{eq:PeriodPolyomialIdentity}
\sum_{\ell = 0}^{k-2} {k - 2 \choose \ell} i^{1-\ell} \Lambda(g; \ell + 1)  N_\chi^{\ell+1} \tau^{\ell} = \frac{4\pi^2 \chi(-1)}{N_\psi^{k} (k-1)}  \res_{z=0}\left( z^{1-k} \omega_\psi(z) \omega_{\chi} \left( z\tau\right)\right).
\end{align}

\section{The order of weak functions}
There is a close relationship between the zero order of a weak function at the origin and the zero behavior of the $L$-series corresponding to its Eisenstein series. Therefore, the aim of this section is to study the spaces 
\begin{align*}  
W_N^{0, \ord \geq \ell} := \{ \omega \in W^0_N \mid \mathrm{ord}_{z=0}(\omega) \geq \ell\} 
\end{align*}
in detail. Note that $W_N^{0,\ord \geq 0} = W_N^0$. The behavior of order encodes important information about $L$-series at critical values. Since we are mainly interested in the order of a weak function in $z=0$, we will $\ord_{z=0}(\omega)$ simply call order of $\omega$ and write $\ord(\omega)$. Although we plan to focus entirely on prime levels later, we will show some results in this section for more general levels if the general case is not harder than the prime case.

\begin{definition} Let $N \geq 3$, $m \geq 0$ be integers. Define the following $(m+1) \times (N-1)$-Vandermonde matrix: 
	\begin{align*}
	\mathrm{CotM}(N,m) := \mat{1 & 1 & 1 & \cdots & 1 \\ \cot\left( \tfrac{\pi}{N}\right) & \cot\left( \tfrac{2\pi}{N}\right) & \cot\left( \tfrac{3\pi}{N}\right) & \cdots & \cot\left( \frac{(N-1)\pi}{N}\right) \\ \cot^2\left( \tfrac{\pi}{N}\right) & \cot^2\left( \tfrac{2\pi}{N}\right) & \cot^2\left( \tfrac{3\pi}{N}\right) & \cdots & \cot^2\left( \frac{(N-1)\pi}{N}\right) \\ \vdots & \vdots & \vdots & \ddots & \vdots \\ \cot^m\left( \tfrac{\pi}{N}\right) & \cot^m\left( \tfrac{2\pi}{N}\right) & \cot^m\left( \tfrac{3\pi}{N}\right) & \cdots & \cot^m\left( \frac{(N-1)\pi}{N}\right)  }.
	\end{align*}
	\end{definition}

In the next proposition we give a formula for the order of a weak function. 

\begin{prop} \label{ordform} Let $\omega \not= 0$ be a weak function of level $N$ that has a removable singulariy in $z=0$, this means $\beta(0) = 0$, i.e., 
	\begin{align*}
	\omega(z) = \sum_{j = 1}^{N-1} \beta^\omega(j) \frac{e(z)}{e\left( \frac{j}{N}\right) - e(z)}. 
	\end{align*}
	Then we have the formula
	\begin{align*}
	\ord(\omega) = \sup \left\{ m \in \N_0 : \mat{\beta^\omega(1) \\ \beta^\omega(2) \\ \vdots \\ \beta^\omega(N-1)} \in \ker(\mathrm{CotM}(N,m))\right\}. 
	\end{align*}
	In the case that $\omega = 0$, this formula formally gives $\ord(\omega) = \infty$.
	\end{prop}

\begin{proof} In \cite{Franke3} the local Taylor expansion of $\omega$ around $z=0$ was calculated in terms of cotangent sums: 
	\begin{align} 
	\nonumber \omega(z) & = \omega(i\infty) - \frac{1}{2} \sum_{j=1}^{N-1} \beta^\omega(j) - \frac{i}{2} \sum_{\nu = 0}^\infty \left( \sum_{u=0}^{\nu+1} \delta_{\nu+1}(u) \sum_{r=1}^{N-1} \beta^\omega(r) \cot^u\left( \frac{\pi r}{N} \right) \right) (z\pi)^\nu \\
	\label{Taylor} & = - \frac{i}{2} \sum_{\nu = 0}^\infty \left( \sum_{u=0}^{\nu+1} \delta_{\nu+1}(u) \sum_{r=1}^{N-1} \beta^\omega(r) \cot^u\left( \frac{\pi r}{N} \right) \right) (z\pi)^\nu.
	\end{align}
	The $\delta_\nu(u)$ are rational numbers that can be calculated explicitely by 
	\begin{align*}
	\delta_{\nu}(u) := \frac{i^{\nu + u}}{(\nu - 1)!} \sum_{\ell=u-1}^{\nu-1} (-1)^{\nu + \ell - u}2^{\nu -1 - \ell} S^*(\nu - 1, \ell)\left( {\ell \choose u} - {\ell \choose u-1}\right).
	\end{align*}
	Here, the numbers $S^*(n,m)$ are related to the Stirling numbers of the second kind $\left\{ {n \atop m} \right\}$ by the equation 
\begin{align*}
S^*(n,m) := m! \left\{ {n \atop m} \right\} = \sum_{j=0}^m (-1)^j {m \choose j} (m - j)^n, \qquad m \leq n.
\end{align*}
	We have $\delta_\nu(\nu) \not= 0$ for all $\nu \geq 1$. Note that $\omega(i\infty) = \sum_{j=1}^{N-1} \beta(j) = 0$ in the case $\omega$ is weak and continuous in $z=0$. Let
	\begin{align*}
	\mu := \sup \left\{ m \in \N_0 : \mat{\beta^\omega(1) \\ \beta^\omega(2) \\ \vdots \\ \beta^\omega(N-1)} \in \ker(\mathrm{CotM}(N,m))\right\}. 
	\end{align*}
	This is equivalent to 
	\begin{align*}
\forall 0 \leq r \leq \mu :	\sum_{j=1}^{N-1} \beta^\omega(j) \cot^r\left( \frac{\pi j}{N} \right) = 0 \quad \text{and} \quad \sum_{j=1}^{N-1} \beta^\omega(j) \cot^{\mu+1}\left( \frac{\pi j}{N}\right) \not= 0.
	\end{align*}
	By \eqref{Taylor} one easily sees that $\mathrm{ord}_{z=0}(\omega(z)) \geq \mu$. On the other hand, when looking at the $\mu$-th coefficient in the Taylor expansion, we obtain
	\begin{align*}
	\omega^{(\mu)}(0) = \frac{\mu! i \pi^\mu}{2} \delta_{\mu+1}(\mu+1) \sum_{j=1}^{N-1} \beta^\omega(j) \cot^{\mu+1}\left(  \frac{\pi j}{N} \right) \not= 0.
	\end{align*}
	This shows $\mathrm{ord}_{z=0}(\omega(z)) \leq \mu$ and we coclude the proposition. 
	\end{proof}

\begin{bem} \label{bem:algorithm} We can use Proposition \ref{ordform} to extract an algorithm to find weak functions of higher order. For fixed level, one only has to find elements in the kernel of a matrix which of course can be achieved after a finite number of steps. One may use the following explicit inverse, provided in \cite{MS}:
	\begin{align*}
	\begin{pmatrix} 1 & 1 & \cdots & 1 & 1 \\ x_1 & x_2 & \cdots & x_{n-1} & x_n \\ x_1^2 & x_2^2 & \cdots & x_{n-1}^2 & x_n^2 \\ \vdots & \vdots & \vdots & \ddots & \vdots \\ x_1^{n-1} & x_2^{n-1} & \cdots & x_{n-1}^{n-1} & x_n^{n-1} \end{pmatrix}^{-1} = (b_{k,\ell})_{1 \leq k, \ell \leq n}
	\end{align*}
	where
	\begin{align*}
	b_{k,\ell} = \frac{(-1)^{n-\ell}}{\prod_{\substack{m=1 \\ m \not= k}}^n (x_k - x_m)} S_{n-\ell}(x_1, x_2, ..., x_{k-1}, x_{k+1}, ..., x_n)
	\end{align*}
	with
	\begin{align*}
	S_{n}(y_1, ..., y_m) := \sum_{1 \leq j_1 < \cdots < j_n \leq m} y_{j_1} y_{j_2}\cdots y_{j_n} \qquad \qquad (0 \leq n \leq m).
	\end{align*}
	\end{bem}
Note that Proposition \ref{ordform} yields the following corollary. 
\begin{cor} \label{dim} We have $\dim_{\C}(W^{0, \mathrm{ord} \geq \ell}_N) = \max\{0, N - \ell - 2\}$. 
	\end{cor}
\begin{pro} By Proposition \ref{ordform} we conclude $W^{0,\mathrm{ord} \geq \ell}_N \cong \mathrm{ker}(\mathrm{CotM}(N,\ell))$ as complex vector spaces. So it suffices to compute the dimension of the latter. This is simple linear algebra. The matrix $\mathrm{CotM}(N,\ell)$ has full rank $r = \min\{N-1, \ell+1\}$, hence we obtain with the rank formula: 
	\begin{align*}
	\dim_{\C}(\mathrm{ker}(\mathrm{CotM}(N,\ell))) = (N-1) - r =  \begin{cases} 0, & \qquad \text{if} \, N-1 \leq \ell+1 \\ N - \ell - 2, & \qquad \text{if} \, N-1 > \ell+1.  \end{cases}
	\end{align*}
	This proves the corollary. 
	\end{pro}

In the following, it is convenient to choose a basis $\alpha^{(N)}_0, \alpha^{(N)}_1, ..., \alpha^{(N)}_{N-3}$ of $W_N^0$ with increasing order, i.e., $\ord(\alpha^{(N)}_j) = j$. We have the following.
\begin{prop} \label{prop:alpha-basis} Let $N$ be a positive integer. Then, there exists a basis $\{ \alpha_j^{(N)}\}_{j=0,...,N-3}$ of the space $W_N^0$, such that the following is satisfied for all $0 \leq j \leq N-3$: 
	\begin{enumerate}[(i)]
		\item We have $\alpha_j^{(N)}(z) = z^j + O(z^{N-2})$ as $z \to 0$, and in particular $\ord(\alpha^{(N)}_j) = j$.
		\item We have $\alpha^{(N)}_j(-z) = (-1)^j \alpha^{(N)}_j(z)$.
	\end{enumerate}
\end{prop}
\begin{proof} 
Obviously,
\begin{align*}
W_N^0 = W_N^{0, \ord \geq 0} \subset W_N^{0, \ord \geq 1} \subset W_N^{0, \ord \geq 2} \subset \cdots \subset W_N^{0,\ord \geq N-3},
\end{align*}
and these spaces have dimensions $N-2, N-1, ..., 1$ by Corollary \ref{dim}. So we can successively add basis vectors $\alpha_j^{(N)}(z)$ that have descending order, i.e., we have a power series expansion 
\begin{align*}
\alpha_j^{(N)}(z) = a_j z^j + a_{j+1}z^{j+1} + \cdots, \qquad a_j \not= 0, 
\end{align*}
Note that we can assume that $\alpha_j^{(N)}(-z) = (-1)^j \alpha_j^{(N)}(z)$, as otherwise we can consider 
\begin{align*}
\widehat{\alpha}_j^{(N)}(z) := \alpha_j^{(N)}(z) + (-1)^j \alpha_j^{(N)}(-z),
\end{align*}
which has again order $j$ but satisfies the desired relation (ii). Of course we can normalize each term to achieve first coefficient 1. We are only left to show that we can achieve $\alpha_j^{(N)}(z) = z^j + O(z^{N-2})$. There is nothing to show in the cases $j \in \{N-3, N-4\}$. For $j = N-5$, we have $\alpha_{N-5}^{(N)}(z) = z^{N-5} + b_1 z^{N-3} + O(z^{N-2})$ for some $b_1 \in \C$, so consider $\alpha_{N-5}^{(N)}(z) - b_1 \alpha_{N-3}^{(N)}(z)$ instead. By inductively continuing this procedure, we obtain the assertion.
\end{proof}
\begin{bem} Note that the algorithm proposed in Remark \ref{bem:algorithm} can be used to calculate the basis vectors $\alpha_j^{(N)}$.
\end{bem}
We give an example.
\begin{bsp} \label{bsp:Basis} Let $N=5$. Put $\mu_1 := \cot(\frac{\pi}{5}) = \sqrt{1+\frac{2}{\sqrt{5}}}$ and $\mu_2 := \cot(\frac{2\pi}{5}) = \sqrt{1-\frac{2}{\sqrt{5}}}$. We find, using 
	\begin{align*}
	\ker \begin{pmatrix} 1 & 1 & 1 & 1 \\ \mu_1 & \mu_2 & - \mu_2 & - \mu_1 \\ \mu_1^2 & \mu_2^2 & \mu_2^2 & \mu_1^2 \end{pmatrix} = \left< \begin{pmatrix} 1 \\  - \frac{\mu_1}{\mu_2} \\ \frac{\mu_1}{\mu_2} \\ -1 \end{pmatrix} \right>
	\end{align*}
	that
	\begin{align*}
	\alpha^{(5)}_2(z) & = \frac{5 i}{4 \pi^2 \sqrt{5+2 \sqrt{5}}} \left( \frac{e(z)}{e(\frac15) - e(z)} - \frac{\mu_1}{\mu_2} \frac{e(z)}{e(\frac25) - e(z)} + \frac{\mu_1}{\mu_2} \frac{e(z)}{e(\frac35) - e(z)} - \frac{e(z)}{e(\frac45) - e(z)}\right) \\
	& = z^2 + O\left(z^4\right).
	\end{align*}
	Similarly, we find
	\begin{align*}
	\alpha^{(5)}_1(z) & = \frac{\sqrt{5}i}{4\pi} \left( \frac{e(z)}{e(\frac15) - e(z)} -  \frac{e(z)}{e(\frac25) - e(z)} - \frac{e(z)}{e(\frac35) - e(z)} + \frac{e(z)}{e(\frac45) - e(z)} \right) \\
	& = z + O\left( z^3 \right), \\
	\alpha^{(5)}_0(z) & = \frac{i}{\sqrt{2 - \frac{2}{\sqrt{5}}}} \left( \frac{e(z)}{e(\frac15) - e(z)} -  \frac{e(z)}{e(\frac25) - e(z)} + \frac{e(z)}{e(\frac35) - e(z)} - \frac{e(z)}{e(\frac45) - e(z)} \right) \\
	& - \frac{i \left(15+\sqrt{5}\right)}{4 \sqrt{5+2 \sqrt{5}}} \left( \frac{e(z)}{e(\frac15) - e(z)} - \frac{\mu_1}{\mu_2} \frac{e(z)}{e(\frac25) - e(z)} + \frac{\mu_1}{\mu_2} \frac{e(z)}{e(\frac35) - e(z)} - \frac{e(z)}{e(\frac45) - e(z)}\right) \\
	& = c_1 \frac{e(z)}{e(\frac15) - e(z)} + c_2 \frac{e(z)}{e(\frac25) - e(z)} - c_2 \frac{e(z)}{e(\frac35) - e(z)} - c_1 \frac{e(z)}{e(\frac45) - e(z)} \\
	& = 1 + O\left(z^4\right),
	\end{align*}
	with
	\begin{align*}
	c_1 & := -i \, \sqrt{\frac{25}{2}-\frac{11 \sqrt{5}}{2}}, \qquad c_2 := i \, \sqrt{\frac{25}{2} +\frac{11 \sqrt{5}}{2}}.
	\end{align*}
	\end{bsp}
The property of ascending order is very useful. With its help we can immediately give bases for the subspaces $W_N^{0, \ord \geq \ell}$ just introduced:
\begin{prop} \label{prop:ordlbasis} Let $0 \leq \ell \leq N-3$. The set $\{ \alpha_j^{(N)}\}_{j=\ell,...,N-3}$ is a basis of $W_N^{0, \ord \geq \ell}$.
	\end{prop}
\begin{proof} Immediate with Corollary \ref{dim} and Proposition \ref{prop:alpha-basis}.
	\end{proof}
Now, we take a closer look at the spaces 
\begin{align} \label{eq:V-definition}
V_{M,N}^{(\ell_1, \ell_2)} := W_M^{0, \ord \geq \ell_1} \otimes W_N^{0, \ord \geq \ell_2}. 
\end{align}
The following is an immediate consequence of Corollary \ref{dim}.

\begin{cor} \label{cor:Vdim} Let $0 \leq \ell_1 \leq M-2$ and $0 \leq \ell_2 \leq N-2$. We then have the formula 
	\begin{align*}
	\dim_\C\left(V_{M,N}^{(\ell_1, \ell_2)}\right) = (M-\ell_1 -2)(N-\ell_2-2).
	\end{align*}
\end{cor}

We can use the elements $\alpha_j^{(M)}$ and $\alpha_j^{(N)}$, to give a basis for $V_{M,N}^{(\ell_1, \ell_2)}$. Ideed, for $0 \leq \ell_1 \leq M-3$ and $0 \leq \ell_2 \leq N-3$, note that
\begin{align} \label{eq:Vbasis}
V_{M,N}^{(\ell_1, \ell_2)} = \bigoplus_{\substack{c \geq \ell_1 \\ d \geq \ell_2}} \C \left(\alpha_c^{(M)} \otimes \alpha_d^{(N)}\right)
\end{align}
with Proposition \ref{prop:ordlbasis} .

While we could easily assign orders to elements from $W_M$ and $W_N$, it is not obvious how we should interpret elements in $V_{M,N}^{(\ell_1, \ell_2)}$ in this respect. One starting point is to assign elementary tensors pairs $(\ord \geq \ell_1, \ord \geq \ell_2)$ and this notion of order in the case of space $V_{M,N}^{(\ell_1, \ell_2)}$ is so far to be understood rather abstractly. The next step is to define orders for arbitrary elements in $W^0_{M} \otimes W^0_{N}$. In the following, we reinterpret it into a more tangible term using complex analysis. Before we can do this, we need the following. 

\begin{definition} \label{def:power-series-orders} Let $\mu$ and $\nu$ be nonngative integers. Define $\C[[z]](\tau)^{(\nu, \mu)}$ to be the vector space of meromorphic functions $f \in \mathcal{M}(\C \times \C)$, such that locally 
	\begin{align*}
	f(z,\tau) = \sum_{j=0}^\infty P_j(\tau)z^j
	\end{align*}
	with polynomials $P_j(\tau)$ (of degree at most $j$), such that $j - \mathrm{deg}(P_j(\tau)) \geq \mu$ and $j - \mathrm{deg}(\tau^j P_j(\frac{1}{\tau})) \geq \nu$ for all $j \in \N_0$.
\end{definition}
For the previous definition one should note that we use the convention $\deg(0) := -\infty$. We also note the following.

\begin{prop} \label{prop:graded-L} The space 
	\begin{align*}
	\mathfrak{L} := \bigoplus_{(\nu, \mu) \in \N_0^2} \C[[z]](\tau)^{(\nu,\mu)}
	\end{align*}
	is a graded algebra via the component multiplication induced by 
	\begin{align*}
	\C[[z]](\tau)^{(\nu_1,\mu_1)} \times \C[[z]](\tau)^{(\nu_2,\mu_2)} & \to \C[[z]](\tau)^{(\nu_1+\nu_2,\mu_1+\mu_2)}, \\
	(f,g) & \mapsto fg.
	\end{align*}
	For $f_1 \in \C[[z]](\tau)^{(\nu_1,\mu_1)}$ and $f_2 \in \C[[z]](\tau)^{(\nu_2,\mu_2)}$ we have 
	\begin{align*} 
	f_1 + f_2 \in \C[[z]](\tau)^{(\min\{ \nu_1, \nu_2\},\min\{\mu_1, \mu_2\})}.
	\end{align*}
\end{prop}
\begin{proof} We only have to show that the multiplication is well defined. For $f(z, \tau) := \sum_{j=0}^\infty P_j(\tau)z^j$ and $g(z, \tau) := \sum_{j=0}^\infty Q_j(\tau)z^j$ we find with standard Cauchy convolution 
	\begin{align*}
	(fg)(z, \tau) = \sum_{j=0}^\infty \left( \sum_{n=0}^j P_{n}(\tau)Q_{j-n}(\tau)\right) z^j.
	\end{align*}
	Obviously, the product series converges again locally around $0$ for all values of $\tau$. Now we have
	\begin{align*}
	j - \deg(P_n(\tau)Q_{j-n}(\tau)) = n + (j-n) - \deg(P_n(\tau)) - \deg(Q_{j-n}(\tau)) \geq \mu_1 + \mu_2,
	\end{align*}
	and the proof of $j - \deg(\tau^j P_n(\frac{1}{\tau}) Q_{j-n}(\frac{1}{\tau})) \geq \nu_1 + \nu_2$ works the same. Since $$\deg\left(\sum_{j=0}^n P_{j}(\tau)Q_{n-j}(\tau)\right) \leq \max_{0 \leq j \leq n}\{ \deg(P_j(\tau)Q_{n-j}(\tau))\},$$ 
	the proof of the first assertion is complete. For the second assertion, consider 
	\begin{align*}
	f_1(z,\tau) + f_2(z,\tau) = \sum_{j=0}^\infty (P_{1,j}(\tau) + P_{2,j}(\tau))z^j.
	\end{align*}
	Note that $j - \deg(P_{1,j}(\tau) + P_{2,j}(\tau)) \geq j - \max\{\deg(P_{1,j}(\tau)), \deg(P_{2,j}(\tau))\} = \min\{j - \deg(P_{1,j}(\tau)), j - \deg(P_{2,j}(\tau))\} = \min\{\mu_1, \mu_2\}$. Similarly, we find $j - \deg(\tau^j (P_{1,j}(\frac{1}{\tau}) + P_{2,j}(\frac{1}{\tau}))) \geq j - \max\{\deg(\tau^j P_{1,j}(\frac{1}{\tau})), \deg(\tau^j P_{2,j}(\frac{1}{\tau}))\} = \min\{\nu_1, \nu_2\}$. 
\end{proof}
We note that there are embeddings $\C[[z]](\tau)^{(\nu_1,\mu_1)} \to \C[[z]](\tau)^{(\nu,\mu)}$ if $\mu \leq \mu_1$ and $\nu_1 \leq \nu$.

Although the definition of the vector space $\C[[z]](\tau)$ and its subspaces $\C[[z]](\tau)^{(\mu, \nu)}$ is useful, it is not yet entirely sufficient to introduce precisely the notion of order in two variables. However, the foundation for the motivation has already been laid:

\begin{definition} \label{def:ordercomplex} We say that a, for all $\tau$ locally convergent, power series
	\begin{align*}
	f(z, \tau) := \sum_{j=0}^\infty P_j(\tau)z^j, \qquad (P_j \in \C[X], \deg(P_j) \leq j)
	\end{align*}
	has order 
	\begin{align*}
	\ord(f) := \left( \min_{j \geq 0} j - \mathrm{deg}\left(\tau^j  P_j\left(\frac{1}{\tau}\right)\right), \min_{j \geq 0} j - \mathrm{deg}(P_j(\tau))  \right).
	\end{align*}
	We also put $\ord(0) := (\infty,\infty)$.
	\end{definition}

It is easy to prove the following important observation. 

\begin{prop} \label{prop:tensor-analytic} Let $f$ and $g$ be globally meromorphic functions analytic in a neighborhood of 0 with $\ord_{z=0}(f(z)) = \mu$ and $\ord_{z=0}(g(z)) = \nu$. Put $h(z,\tau) := f(z)g(z\tau)$. Then we have $\ord(h) = (\nu, \mu)$. In particular, $h \in \C[[z]](\tau)^{(\nu, \mu)}$. 
\end{prop}
\begin{proof} Using the Taylor expansions $f(z) := \sum_{n=0}^\infty a_n z^n$ and $g(w) := \sum_{m=0}^\infty b_m w^m$ we find 
	\begin{align*}
	f(z)g(z\tau) = \left( \sum_{n=0}^\infty a_n z^n \right) \left( \sum_{m=0}^\infty b_m \tau^m z^m \right) = \sum_{j=0}^\infty \left( \sum_{n=0}^j a_{j-n} b_n \tau^n \right)z^j.
	\end{align*}
	Since by assumption $a_0 = a_1 = \cdots = a_{\mu-1} = 0$ and $b_0 = b_1 = \cdots = b_{\nu-1} = 0$, but $a_{\mu} b_{\nu} \not= 0$ we find for all $j \geq 0$:
	\begin{align*}
	\deg\left( \sum_{n=0}^j a_{j-n} b_n \tau^n\right) =  \deg\left( \sum_{n=0}^{j-\mu} a_{j-n} b_n \tau^n\right)  \leq  j - \mu
	\end{align*}
	and 
	\begin{align*}
	\deg\left( \sum_{j=0}^n a_{j-n} b_n \tau^{j-n}\right) = \deg\left( \sum_{j=0}^{j-\nu} a_{j-n} b_n \tau^{j-n}\right) \leq j - \nu.
	\end{align*}
	Now choose $j := \nu + \mu \geq 0$. In this particular case, we find
	\begin{align*}
	\deg\left( \sum_{n=0}^{\nu+\mu} a_{\nu+\mu-n} b_n \tau^n\right) = \deg\left( a_\mu b_\nu \tau^{\nu} + \sum_{n=0}^{\nu-1} a_{\nu+\mu-n} b_n \tau^n\right) = \nu = \nu + \mu - \mu.
	\end{align*}
	Similarly,
		\begin{align*}
	\deg\left( \sum_{n=0}^{\nu+\mu} a_{\nu+\mu-n} b_n \tau^{\nu+\mu-n}\right) = \deg\left( a_\mu b_\nu \tau^{\mu} + \sum_{n=0}^{\mu-1} a_{\nu+\mu-n} b_n \tau^{\nu+\mu-n}\right) = \mu = \mu + \nu - \nu.
	\end{align*}
	This proves $\ord(h) = (\nu, \mu)$, and clearly $h \in \C[[z]](\tau)^{(\nu, \mu)}$.
\end{proof}
The following approach is motivated by Proposition \ref{prop:tensor-analytic}. We are able to construct meromorphic functions of two variables $z$ and $\tau$ from the data $\omega \otimes \eta$ as follows: On elementary tensors define a linear map $\Xi_{M,N}$ by
\begin{align}
\label{map} \omega \otimes \eta \longmapsto \eta(z)\omega(z\tau),
\end{align}
and of course this extends to a linear map $W^0_M \otimes W_N^0 \rightarrow \C[[z]](\tau)$ by Proposition \ref{prop:tensor-analytic}. Our next goal is to prove that this map is injective, and hence, that we do not lose information when going from the spaces $W_M^0 \otimes W_N^0$ to meromorphic  functions. To do so, it is useful to introduce a notion for the subspace of $\C[[z]](\tau)$ genrated by weak pairs $\omega \otimes \eta$. We put
\begin{align*}
J^0_{M,N} := \left< \eta(z)\omega(z\tau) \ \Big{|} \ \omega \otimes \eta \in W_M^0 \otimes W_N^0 \right>.
\end{align*}
Also put 
\begin{align*}
J_{M,N,k}^0 :=  \left< \eta(z)\omega(z\tau) \ \Big{|} \ \omega \otimes \eta \in \left(W_M^0 \otimes W_N^0\right)_k \right>
\end{align*}
and denote the corresponding restricted map $(W_M^0 \otimes W_N^0)_k \to J_{M,N,k}^0$ by $\Xi_{M,N,k}$. 
We can identify both spaces. 
\begin{prop} \label{prop:injective} Let $(a_{c,d})_{0 \leq c \leq M-3, 0\leq d \leq N-3}$ be arbitrary complex numbers and
	\begin{align*}
	\mathcal{M} := \{ (c,d) \in \N_0^2 \colon 0 \leq c \leq M-3, c \leq d \leq c + N- 3\}.
	\end{align*}
	 Then we have 
	\begin{align*}
	\sum_{\substack{0 \leq c \leq M-3 \\ 0 \leq d \leq N-3}} a_{c,d} \alpha^{(N)}_d(z) \alpha_{c}^{(M)}(z\tau) = \sum_{\substack{0 \leq c \leq M-3 \\ c \leq d \leq c+N-3}} a_{c,d-c} \tau^c z^d + \sum_{(n,m) \in \N_0^2 \setminus \mathcal{M}} c_{m,n}\tau^m z^n
	\end{align*}
	for some complex $c_{m,n}$. In particular, the linear maps $\Xi_{M,N} \colon W^0_M \otimes W_N^0 \rightarrow J^0_{M,N}$ and $\Xi_{M,N,k} \colon (W^0_M \otimes W_N^0)_k \rightarrow J^0_{M,N,k}$ described by \eqref{map} are both isomorphisms.
\end{prop}
\begin{proof} The first part is a simple application of the identity theorem for power series. By construction we have for $0 \leq c \leq M-3$ and $0 \leq n \leq N-3$
	\begin{align*}
	\alpha^{(N)}_d(z) \alpha_{c}^{(M)}(z\tau) & = \left( \tau^c z^c + \sum_{j \geq M-2} a_{j} \tau^j z^j\right)\left( z^d+ \sum_{j \geq N-2} b_j z^j\right) \\
	& = \tau^c z^{c+d} + \sum_{(m,n) \in \N_0^2 \setminus \mathcal{M}} r_{m,n} \tau^m z^n
	\end{align*}
	for some complex numbers $r_{m,n}$. Since we have the bijection $\{0, \ldots, M-3\} \times \{0, \ldots, N-3\} \to \mathcal{M}$ with $(c,d) \mapsto (c,c+d)$ the claim now follows with the Identity theorem for power series. Again by the identity theorem the claimed isomorphisms hold, as the monomials $\tau^c z^d$ are linearly independent. 
\end{proof}
\begin{bem} This argument, built on the elementary theory of rational functions, can be truncated by means of Eisenstein series. This concerns in particular the case of prime $M$ and $N$, since the isomorphism from Proposition \ref{prop:prime-injective} factorizes over $J_{M,N}^0$. 
\end{bem}
\begin{cor} \label{cor:p1p2iso} Let $p_1$ and $p_2$ be odd primes and $k \geq 3$. Then we have an isomorphism $J_{p_1,p_2,k}^0 \to \mathcal{E}_k(\Gamma_1(p_1p_2))_0^{(p_1,p_2)}$ given by $\eta(z) \omega(z\tau) \mapsto \vartheta_k(\omega \otimes \eta; p_2\tau)$.
\end{cor}
\begin{proof} Immediate with Propositions \ref{prop:prime-injective} and \ref{prop:injective}.
\end{proof}

The following is a simple, but useful observation. 
\begin{prop} \label{prop:order-equality} Let $\mathcal{S} \subset \{0, \ldots, M-3\} \times  \{0, \ldots, N-3\}$ be non-empty and $(a_{c,d})_{(c,d) \in \mathcal{S}}$ be a family of non-zero complex numbers. Then we have the formula
	\begin{align*}
	\ord\left( \sum_{(c,d) \in \mathcal{S}} a_{c,d} \alpha_d^{(N)}(z)\alpha_{c}^{(M)}(z\tau)\right) = \left(\min_{(c,d) \in \mathcal{S}} c, \min_{(c,d) \in \mathcal{S}} d\right).
	\end{align*}
\end{prop}
\begin{proof} Note that $\ord(\alpha_{n}^{(N)}(z)\alpha_m^{(M)}(z\tau)) = (m,n)$ by Proposition \ref{prop:tensor-analytic}. Applying Proposition \ref{prop:graded-L} multiple times, we conclude 
	\begin{align*}
	\ord\left( \sum_{(c,d) \in \mathcal{S}} a_{c,d} \alpha_d^{(N)}(z)\alpha_{c}^{(M)}(z\tau)\right) \geq \left(\min_{(c,d) \in \mathcal{S}} c, \min_{(c,d) \in \mathcal{S}} d\right)
	\end{align*}
	componentwise. On the other hand, we find with Proposition \ref{prop:injective} 
	\begin{align*}
	\sum_{(c,d) \in \mathcal{S}} a_{c,d} \alpha_d^{(N)}(z)\alpha_{c}^{(M)}(z\tau) = \sum_{\substack{0 \leq c \leq M-3 \\ c \leq d \leq c+N-3 \\ (c,d-c) \in \mathcal{S}}} a_{c,d-c} \tau^c z^d + \sum_{(n,m) \in \N_0^2 \setminus \mathcal{M}} c_{m,n}\tau^m z^n.
	\end{align*}
	In particular, there are non-trivial monomials of the form $a_{\min_{(c,\ell) \in \mathcal{S}} c, d_1 - \min_{(c,\ell) \in \mathcal{S}} c}\tau^{\min_{(c,\ell) \in \mathcal{S}} c} z^{d_1}$ and $a_{c,\min_{(\ell,d) \in \mathcal{S}} d}\tau^{c_1}z^{(\min_{(\ell,d) \in \mathcal{S}} d) + c_1}$, where $\min_{(c,\ell) \in \mathcal{S}} c \leq d_1 \leq \min_{(c,\ell) \in \mathcal{S}} c + N - 3$ and $0 \leq c_1 \leq M-3$ are fixed integers. Writing $\sum_{(c,d) \in \mathcal{S}} a_{c,d} \alpha_d^{(N)}(z)\alpha_{c}^{(M)}(z\tau) = \sum_{j=0}^\infty P_j(\tau)z^j$, we note that
	\begin{align*}
	P_{d_1}(\tau) & = a_{\min_{(c,\ell) \in \mathcal{S}} c, d_1 - \min_{(c,\ell) \in \mathcal{S}} c}\tau^{\min_{(c,\ell) \in \mathcal{S}} c} + \sum_{j=\min_{(c,\ell) \in \mathcal{S}} c + 1}^{d_1} a_j \tau^j, \\
	P_{(\min_{(\ell,d) \in \mathcal{S}} d) + c_1}(\tau) & = a_{c,\min_{(\ell,d) \in \mathcal{S}} d}\tau^{c_1}z^{(\min_{(\ell,d) \in \mathcal{S}} d) + c_1} + \sum_{j=0}^{c_1-1} b_j \tau^j
		\end{align*}
	for some complex $a_j$ and $b_j$, as otherwise the minimality is violated. We conclude, again componentwise, 
	\begin{align*}
	& \left( \min_{j \geq 0} j - \mathrm{deg}\left(\tau^j  P_j\left(\frac{1}{\tau}\right)\right), \min_{j \geq 0} j - \mathrm{deg}(P_j(\tau))  \right) \\
	\leq & \left( d_1 - \deg\left(\tau^{d_1} \left( a_{\min_{(c,\ell) \in \mathcal{S}} c, d_1 - \min_{(c,\ell) \in \mathcal{S}} c}\tau^{-\min_{(c,\ell) \in \mathcal{S}} c} + \sum_{j=\min_{(c,\ell) \in \mathcal{S}} c + 1}^{d_1} a_j \tau^{-j}\right)\right), \right. \\
	& \left. \left(\min_{(\ell,d) \in \mathcal{S}} d\right) + c_1 - \deg\left( a_{c,\min_{(\ell,d) \in \mathcal{S}} d}\tau^{c_1}z^{(\min_{(\ell,d) \in \mathcal{S}} d) + c_1} + \sum_{j=0}^{c_1-1} b_j \tau^j\right)\right) \\
	=& \left(\min_{(c,d) \in \mathcal{S}} c, \min_{(c,d) \in \mathcal{S}} d\right).
	\end{align*}
This proves the claim. 
\end{proof}

We can now define an order on the space $W_M^0 \otimes W_N^0$ by analogy with the previous results. 
\begin{definition} \label{def:oder} Let 
\begin{align*}
0 \not= f := \sum_{\substack{0 \leq m \leq M - 3 \\ 0 \leq n \leq N- 3}} a_{m,n} \left( \alpha_m^{(M)} \otimes \alpha_n^{(N)}\right).
\end{align*}
Then we define
\begin{align*}
\ord(f) := \left( \min \{ 0 \leq m \leq M-3 \ \colon \ \exists \ 0 \leq n \leq N-3 \text{ with } a_{m,n} \not= 0\}, \right. \\
\left. \min \{ 0 \leq n \leq N-3 \ \colon \ \exists \ 0 \leq m \leq M-3 \text{ with } a_{m,n} \not= 0\}\right).
\end{align*}
We also put $\ord(0) := (\infty,\infty)$. 
\end{definition}
Note that if the order is finite in one component, it is also finite in the other one.  
\begin{prop} \label{prop:order-preserv} For all $F \in W_M^0 \otimes W_N^0$, we have $\ord(F) = \ord(\Xi_{M,N}(F))$, where $\Xi_{M,N}$ is the isomorphism in Proposition \ref{prop:injective}.
	\end{prop}
\begin{proof} This is immediate with Proposition \ref{prop:order-equality} and Definition \ref{def:oder}. 
	\end{proof}
The following theorem summarizes the above discussion in a useful fact that establishes a simple connection between the abstract order on $W_M^0 \otimes W_N^0$ and the zero order known from complex analysis. 

\begin{thm} \label{thm:iso} Let $0 \leq \ell_1 \leq M-3$ and $0 \leq \ell_2 \leq N-3$. There is an order and sign preserving isomorphism given by
	\begin{align} \label{eq:V-map}
	\Xi_{M,N}^{(\ell_1, \ell_2)} \colon V_{M,N}^{(\ell_1, \ell_2)}  & \overset{\sim}{\longrightarrow} J_{M,N}^0 \cap \C[[z]](\tau)^{(\ell_1, \ell_2)} \\
	\nonumber \omega \otimes \eta & \longmapsto ((z, \tau) \mapsto \eta(z)\omega(z\tau))
	\end{align}
	on elementary tensors between vector spaces induced by \eqref{map}, where $V_{M,N}^{(\ell_1, \ell_2)}$ is defined in \eqref{eq:V-definition}. 
	\end{thm}
\begin{proof} It is clear that the map is sign preserving. By construction and Definition \ref{def:power-series-orders} the space $J_{M,N}^0 \cap \C[[z]](\tau)^{(\ell_1, \ell_2)}$ contains all functions $(z,\tau) \mapsto f(z,\tau) \in J_{M,N}^0 $ staisfying $\ord(f) \geq (\ell_1, \ell_2)$. Now the claim is immediate with Propositions and \ref{prop:order-equality} and \ref{prop:order-preserv}.
	\end{proof}

For our investigations, we still lack a means to infer vanishing coefficients of polynomials in the power series expansion in $\alpha_n^{(N)}(z)\alpha_m^{(M)}(z\tau)$ from properties of the basis vectors $\alpha_m^{(M)} \otimes \alpha_n^{(N)}$ for $W_M^0 \otimes W_N^0$. Let $T \geq 0$ be an integer. For the polynomial spaces
\begin{align*}
\C[\tau]^{T \geq \deg} := \{ P \in \C[\tau] \colon \deg(P) \leq T\}
\end{align*} 
and subsets $\mathcal{S} \subset \{0,...,T\}$ we consider linear maps 
\begin{align*}
\mathrm{Coeff}_{\mathcal{S}, T} \colon \C[z]^{T \geq \deg} & \longrightarrow \C^{|\mathcal{S}|} \\
\sum_{j=0}^T a_j \tau^j  & \longmapsto (a_j)_{j \in \mathcal{S}}.
\end{align*}
To extract the polynomials $P_j(\tau)$ in the power series expansions of $\eta(z)\omega(z\tau)$, we can use the residue map
\begin{align*}
\mathrm{Res}_T \colon J^0_{M,N} &\longrightarrow \C[\tau]^{T \geq \deg},  \\
\eta(z)\omega(z\tau) & \longmapsto \res_{z=0}\left( z^{-(T+1)} \eta(z) \omega(z\tau)\right).
\end{align*}
We can prove the following.
\begin{thm} \label{thm:exactsequence} Let $M$, $N$ and $T$ be integers. Let $\ell_1$ and $\ell_2$ be integers such that $0 \leq \ell_1 \leq M-2$, $0 \leq \ell_2 \leq N-2$, and $\ell_1 + \ell_2 \leq T+1$. Choose 
	\begin{align*}
	\mathcal{S} \subset \{0, 1, ..., \ell_1-1\} \cup \{T+1-\ell_2, T+2-\ell_2, \ldots, T\} \subset \{0, \ldots T\}.
	\end{align*}
	Then we have 
	\begin{align} \label{eq:kernel-sub}
	\bigoplus_{\substack{\left( m \notin \mathcal{S} \wedge T-n \notin \mathcal{S}\right) \\ \vee \left( (m,T-n) \in \mathcal{S}^2 \wedge  m \not= T-n \right)}} \C \left( \alpha_m^{(M)} \otimes \alpha_n^{(N)}\right) \subset \ker\left( \mathrm{Coeff}_{\mathcal{S}, T} \circ \mathrm{Res}_T \circ \Xi_{M,N}\right).
	\end{align}
	If additionally $T \leq \min\{M-3,N-3\}$, then, for any $\mathcal{S} \subset \{0,1,\ldots,T\}$, we have the exact sequence 
	\begin{align} \label{eq:exactseq}
	0 \longrightarrow \bigoplus_{m \notin \mathcal{S} \vee m \not= T-n} \C \left( \alpha_m^{(M)} \otimes \alpha_n^{(N)}\right) \overset{\Xi_{M,N}}{\longrightarrow} J_{M,N}^0 \overset{\mathrm{Coeff}_{\mathcal{S}, T} \circ \mathrm{Res}_T }{\longrightarrow }\C^{|\mathcal{S}|} \longrightarrow 0.
	\end{align}
\end{thm}
\begin{proof} During the proof we write
	\begin{align*}
	\alpha_{m}^{(M)}(z) = \sum_{\ell=0}^\infty a_{m,\ell}^{(M)}z^\ell = \sum_{\ell=0}^{M-3}\delta_{m,\ell} z^\ell + O\left(z^{M-2}\right)
	\end{align*}
	where the last equality follows from Proposition \ref{prop:alpha-basis}. We first show \eqref{eq:kernel-sub}. Recall that
	\begin{align*}
    \Xi_{M,N}\left( \alpha_m^{(M)} \otimes \alpha_n^{(N)}\right) = \alpha_n^{(N)}(z)\alpha_m^{(M)}(z\tau) = \sum_{\ell=0}^T a_{m,\ell}^{(M)} a_{n,T-\ell}^{(N)} \tau^\ell.
	\end{align*}
	Applying $\mathrm{Coeff}_{\mathcal{S}, T} \circ \mathrm{Res}_T$ this equals
	\begin{align*}
	\left( a_{m,\ell}^{(M)} a_{n,T-\ell}^{(N)} \right)_{\ell \in \mathcal{S}}.
	\end{align*}
	If we assume $m \notin \mathcal{S}$, we find for all $0 \leq \ell \leq \ell_1 - 1 $ and $\ell \in \mathcal{S}$ that $a_{m,\ell}^{(M)} a_{n,T-\ell}^{(N)} = \delta_{m,\ell} a_{n,T-\ell}^{(N)} = 0$ (note that $\ell_1 -1 \leq M-3$ by assumption). Likewise, assuming $T - n \notin \mathcal{S}$, we obtain for all $\ell \in \mathcal{S}$ with $T + 1 - \ell_2 \leq \ell \leq T$ the equality $a_{m,\ell}^{(M)} a_{n,T-\ell}^{(N)} = a_{m,\ell}^{(M)} \delta_{n,T-\ell} = 0$, since $n = T - \ell$ implies $\ell = T - n \in \mathcal{S}$, which contradicts our assumption. 
	Now assume $(m,T-n) \in \mathcal{S}^2$ and $m \not= T-n$. In this case we find $a_{m,\ell}^{(M)} a_{n,T-\ell}^{(N)} = \delta_{m,\ell} \delta_{n,T-\ell} = 0$. This proves \eqref{eq:kernel-sub}. \\
	Next we show that for all $\mathcal{S} \subset \{0,1,\ldots, T\}$ the sequence \eqref{eq:exactseq} is exact if we additionally assume $T \leq \min\{M-3, N-3\}$. Note that the resctriction of $\Xi_{M,N}$ is injective by Proposition \ref{prop:injective}. Under the restriction $T \leq \min\{M-3, N-3\}$ we also obtain 
	\begin{align*}
	\left( a_{m,\ell}^{(M)} a_{n,T-\ell}^{(N)}\right)_{\ell \in \mathcal{S}} = \left( \delta_{m,\ell} \delta_{n,T-\ell}\right)_{\ell \in \mathcal{S}}.
	\end{align*}
	For arbitrary $\ell \in \mathcal{S}$, note that this is not the zero vector if and only if there is some $\ell \in \mathcal{S}$ such that $\ell = m$ and $T - \ell = n$, i.e., $m = T-n$ and $m \in \mathcal{S}$. Consequently, it is zero if and only if $m \notin \mathcal{S}$ or $T - m \not= n$. This proves by linear independency of the non-zero components 
	\begin{align*}
	\Xi_{M,N}\left( \bigoplus_{m \notin \mathcal{S} \vee m \not= T-n} \C \left( \alpha_m^{(M)} \otimes \alpha_n^{(N)}\right) \right) = \ker\left( \mathrm{Coeff}_{\mathcal{S}, T} \circ \mathrm{Res}_T  \right).
	\end{align*}
	Finally, for any $(\lambda_\ell)_{\ell \in \mathcal{S}} \in \C^{|\mathcal{S}|}$ we can find the pre-image $\sum_{\ell \in \mathcal{S}} \lambda_\ell \alpha_{T-\ell}^{(N)}(z)\alpha^{(M)}_{\ell}(z\tau) \in J_{M,N}^0$ which proves that $\mathrm{Coeff}_{\mathcal{S}, T} \circ \mathrm{Res}_T \colon J_{M,N}^0 \to \C^{|\mathcal{S}|}$ is onto. The theorem is proved.
	\end{proof}
In the context of modular forms we need the above theorem with respect to a fixed weight $k$. As we have introduced the spaces $(W_M^0 \otimes W_N^0)_k$ in order to eleminate trivial canceling when going from rational functions to modular forms, we should enure that Theorem \ref{thm:exactsequence} still works ``when restricting to $(W_M^0 \otimes W_N^0)_k$''. 
\begin{cor} \label{cor:exactsequence-k} Let $M$, $N$ and $T$ be integers. Let $\ell_1$ and $\ell_2$ be integers such that $0 \leq \ell_1 \leq M-2$, $0 \leq \ell_2 \leq N-2$, and $\ell_1 + \ell_2 \leq T+1$. Choose 
	\begin{align*}
	\mathcal{S} \subset \{0, 1, ..., \ell_1-1\} \cup \{T+1-\ell_2, T+2-\ell_2, \ldots, T\} \subset \{0, \ldots T\}.
	\end{align*}
	Then we have 
	\begin{align*} 
	\bigoplus_{\substack{\left( m \notin \mathcal{S} \wedge T-n \notin \mathcal{S}\right) \\ \vee \left( (m,T-n) \in \mathcal{S}^2 \wedge  m \not= T-n \right) \\ m+n \equiv T \pmod{2}}} \C \left( \alpha_m^{(M)} \otimes \alpha_n^{(N)}\right) \subset \ker\left( \mathrm{Coeff}_{\mathcal{S}, T} \circ \mathrm{Res}_T \circ \Xi_{M,N,T}\right).
	\end{align*}
	If additionally $T \leq \min\{M-3,N-3\}$, then, for any $\mathcal{S} \subset \{0,1,\ldots,T\}$, we have the exact sequence 
	\begin{align*} 
	0 \longrightarrow \bigoplus_{\substack{m \notin \mathcal{S} \vee m \not= T-n \\ m+n \equiv T \pmod{2}}} \C \left( \alpha_m^{(M)} \otimes \alpha_n^{(N)}\right) \overset{\Xi_{M,N,T}}{\longrightarrow} J_{M,N,T}^0 \overset{\mathrm{Coeff}_{\mathcal{S}, T} \circ \mathrm{Res}_T }{\longrightarrow }\C^{|\mathcal{S}|} \longrightarrow 0.
	\end{align*}
	\end{cor}
\begin{proof} As we have $\alpha_m^{(M)} \otimes \alpha_{n}^{(N)} \in (W_M^0 \otimes W_N^0)_T$ if and only if $m+n \equiv T \pmod{2}$, the first claim follows directly with Theorem \ref{thm:exactsequence}. A similar observation holds for the exact sequence, as the pre-images of $\mathrm{Coeff}_{\mathcal{S}, T} \circ \mathrm{Res}_T $ are part of the subspace $J_{M,N,T}^0$.
	\end{proof}
\section{Applications to $L$-series of Eisenstein series}

Fix a weight $k \geq 3$. Depending on that $k$ is odd or even, we obtain different kernels of the map 
\begin{align*}
\omega \otimes \eta \longmapsto \vt_k(\omega \otimes \eta; \tau) := -2\pi i\sum_{x \in \Q^\times} \res_{z=x} \left( z^{k-1} \eta(z) \omega(z\tau)\right).
\end{align*}
So we only focus on elements with the right sign. A critical tool for the main proofs is Eichler duality, so Fourier transforms will play a significant role. For the sake of clarity and the convenience of the reader, we work these out very explicitly. First, we put $M=p_1$ and $N=p_2$ with odd prime numbers $p_1$ and $p_2$, and write the elements $\alpha_c^{(p_1)}$ and $\alpha_d^{(p_2)}$ (with $0 \leq c \leq p_1-2$, $0 \leq d \leq p_2-2$) as linear combinations in the $\omega_\chi$:
\begin{align} \label{eq:original-alpha}
\alpha_c^{(p_1)} = \sum_{\chi \in \mathcal{C}_0^{\mathrm{prim}}(p_1)} a_{\chi}(c) \omega_\chi, \qquad \alpha_d^{(p_2)} = \sum_{\psi \in \mathcal{C}_0^{\mathrm{prim}}(p_2)} a_{\psi}(d) \omega_\psi.
\end{align}
Note that this is possible as for primes all non-principal characters are primitive. Corresponding to this, we now consider the change of basis $\alpha_c^{(p_1)} \otimes \alpha_d^{(p_2)} \mapsto \widetilde{\alpha}_c^{(p_1)} \otimes \widehat{\alpha}_d^{(p_2)}$ defined by
\begin{align} \label{eq:dual-alpha}
\widetilde{\alpha}_c^{(p_1)} := \sum_{\chi \in \mathcal{C}_0^{\mathrm{prim}}(p_1)} \chi(-1) \cG(\chi) a_{\chi}(c) \omega_{\overline{\chi}}, \qquad \widehat{\alpha}_d^{(p_2)} := \sum_{\psi \in \mathcal{C}_0^{\mathrm{prim}}(p_2)} \cG(\psi) a_{\psi}(d) \omega_{\overline{\psi}},
\end{align}
where $\mathcal{G}(\chi)$, $\mathcal{G}(\psi)$ are the Gauss sums of $\chi$, $\psi$ and the numbers $a_{\chi}(c)$ and $a_{\psi}(d)$ are defined in \eqref{eq:original-alpha}. We can show that these elements again give a basis.
\begin{lemma} \label{lemma:basis-change} The elements in \eqref{eq:dual-alpha} give a basis of $W_{p_1}^0$ and $W_{p_2}^0$, respectively. 
	\end{lemma}
\begin{proof} We only give a proof for $p_1$ as the case $p_2$ is similar. Write $\{\chi_1, ..., \chi_{p_1-2}\}$ for the non-principal characters modulo $p_1$. Put
	\begin{align*}
	A & := \begin{pmatrix} a_{\chi_1}(0) & a_{\chi_2}(0) & \cdots & a_{\chi_{p_1-2}}(0) \\ a_{\chi_1}(1) & a_{\chi_2}(1) & \cdots & a_{\chi_{p_1-2}}(1) \\ \vdots & \ddots & \vdots & \vdots \\ a_{\chi_1}(p_1-3) & a_{\chi_2}(p_1-3) & \cdots & a_{\chi_{p_1-2}}(p_1-3) \end{pmatrix}, \\
	B & := \begin{pmatrix} \chi_1(-1) \mathcal{G}(\chi_1) & 0& \cdots & 0\\ 0& \chi_2(-1)\mathcal{G}(\chi_2) & \cdots & 0 \\ \vdots & \ddots & \vdots & \vdots \\ 0 & 0 & \cdots & \chi_{p_1-2}(-1)\mathcal{G}(\chi_{p_1-2}). \end{pmatrix}.
	\end{align*}
Note that \eqref{eq:original-alpha} implies 
\begin{align*}
A \begin{pmatrix} \omega_{\chi_1} \\ \omega_{\chi_2} \\ \vdots \\ \omega_{\chi_{p_1-2}}\end{pmatrix} =  \begin{pmatrix} \alpha_0^{(p_1)} \\ \alpha_1^{(p_1)} \\ \vdots \\ \alpha_{p_1-3}^{(p_1)}\end{pmatrix}.
\end{align*}
Since this is a change of basis, $A$ is regular. The same holds for $B$, as all involved characters are primitive. Let $S$ be the swapping matrix sending $(\omega_{\chi})$ to $(\omega_{\overline{\chi}})$. With \eqref{eq:dual-alpha} we conclude that
\begin{align*}
ABS \begin{pmatrix} \omega_{\chi_1} \\ \omega_{\chi_2} \\ \vdots \\ \omega_{\chi_{p_1-2}}\end{pmatrix} = \begin{pmatrix} \widetilde{\alpha}_0^{(p_1)} \\ \widetilde{\alpha}_1^{(p_1)} \\ \vdots \\ \widetilde{\alpha}_{p_1-3}^{(p_1)}\end{pmatrix}.
\end{align*}
Since $ABS$ is regular, this is again a change of basis, and the claim follows. 
\end{proof}
\begin{bsp} \label{bsp:Basiswechsel} Let $\chi_5$ be the Dirichlet character modulo $5$ satisfying $\chi_5(2) = i$. We then have with Example \ref{bsp:Basis}
	\begin{align*}
	\alpha_2^{(5)} = -\frac{\sqrt[4]{3-4 i} 5^{\frac{3}{4}}}{4 \sqrt{2} \pi ^2} \omega_{\chi_5} + \frac{\sqrt[4]{3+4 i} 5^{\frac{3}{4}}}{4 \sqrt{2} \pi ^2} \omega_{\overline{\chi_5}},
	\end{align*}
	where all roots are taken in the principal branch. We can use this to compute $\widetilde{\alpha}_2^{(5)}$ and $\widehat{\alpha}_2^{(5)}$ explicitely. We have $\mathcal{G}(\chi_5) = i \sqrt[4]{-15+20 i}$ and $\mathcal{G}(\overline{\chi_5}) = i \sqrt[4]{-15-20 i}$. Hence,
	\begin{align*}
	\widetilde{\alpha}_2^{(5)} = \frac{5 i \sqrt[4]{7+24 i}}{4 \sqrt{2} \pi ^2} \omega_{\overline{\chi_5}} - \frac{5 i \sqrt[4]{7-24 i}}{4 \sqrt{2} \pi ^2} \omega_{\chi_5}, \quad \widehat{\alpha}_2^{(5)} = -\frac{5 i \sqrt[4]{7+24 i}}{4 \sqrt{2} \pi ^2} \omega_{\overline{\chi_5}} + \frac{5 i \sqrt[4]{7-24 i}}{4 \sqrt{2} \pi ^2} \omega_{\chi_5}.
	\end{align*}
	\end{bsp}
In the natural mapping from rational functions to modular forms we can interpose just this change of basis. We define 
	\begin{align} 
	\label{eq:tildetheta} \widetilde{\vartheta}_k \colon J_{p_1,p_2,k}^0 & \longrightarrow \mathcal{E}_k(\Gamma_1(p_1 p_2))_0^{(p_1, p_2)}  \\
\nonumber \alpha_{d}^{(p_2)}(z)\alpha_{c}^{(p_1)}(z\tau) & \longmapsto \vt_k\left(\widetilde{\alpha}_c^{(p_1)}\otimes \widehat{\alpha}_d^{(p_2)}; p_2 \tau\right).
\end{align}
Note that this map is an isomorphism by Corollary \ref{cor:p1p2iso} and Lemma \ref{lemma:basis-change}.

We have the following key theorem, that describes spaces of Eisenstein series with $L$-series vanishing at specific critical values precisely for small weights.  

\begin{thm} \label{thm:mainsmallweight} Let $p_1, p_2$ be odd primes, $3 \leq k \leq \min\{p_1-2,p_2-2\}$ an integer, and $0 \leq \ell_1 \leq p_1-2$, $0 \leq \ell_2 \leq p_2-2$, with $\ell_1 + \ell_2 \leq k-1$. Let $\mathcal{S} \subset \{0, \ldots, k-2\}$ be an arbitrary subset. Consider the homomorphism $\xi_{p_1, p_2, k} : (W_{p_1}^0 \otimes W_{p_2}^0)_k \to \mathcal{E}_k(\Gamma_1(p_1p_2))_0^{ (p_1, p_2)}$ defined by 
	\begin{align*}
	\alpha_c^{(p_1)} \otimes \alpha_d^{(p_2)} \longmapsto \vartheta_k\left(\widetilde{\alpha}_c^{(p_1)} \otimes \widehat{\alpha}_d^{(p_2)} ; p_2\tau\right).
	\end{align*}
	Then we have an exact sequence 
	\begin{align*}
		0 \longrightarrow \bigoplus\limits_{\substack{m \notin \mathcal{S} \vee m \not= k-2-n \\ m+n \equiv k \pmod{2}}} \C \left( \alpha_m^{(p_1)} \otimes \alpha_n^{(p_2)}\right)  \overset{\xi_{p_1,p_2,k}}{\longrightarrow} \mathcal{E}_k(\Gamma_1(p_1p_2))_0^{ (p_1, p_2)} \overset{\mathcal{L}_{\mathcal{S},k}}{\longrightarrow }\C^{|\mathcal{S}|} \longrightarrow 0,
	\end{align*}
	where the linear map $\mathcal{L}_{\mathcal{S},k} \colon \mathcal{E}_k(\Gamma_1(p_1p_2))_0^{ (p_1, p_2)} \to \C^{|\mathcal{S}|}$ is given by
	\begin{align*}
	\mathcal{L}_{\mathcal{S},k}(f) :=  \frac{(-2\pi i)^k}{(k-2)!} \frac{p_2^{k-1}}{4\pi^2} \left( \binom{k-2}{\ell} i^{1-\ell} (2\pi)^{-\ell-1} \Gamma(\ell+1) L\left( f; \ell+1\right) p_1^\ell \right)_{\ell \in \mathcal{S}}.
	\end{align*}
	\end{thm}
\begin{proof} By Corollary \ref{cor:exactsequence-k}, putting $M := p_1$, $N := p_2$ and $T := k-2$, it suffices to show that the diagram
	\[
	\begin{tikzcd} 
	0 \arrow[r] & \bigoplus\limits_{\substack{m \notin \mathcal{S} \vee m \not= k-2-n \\ m+n \equiv k \pmod{2}}} \C \left( \alpha_m^{(p_1)} \otimes \alpha_n^{(N)}\right)  \arrow[r,"\Xi_{p_1,p_2, k-2}"' swap] \arrow[rd, "\xi_{p_1, p_2, k}"'] & J_{p_1,p_2,k}^0 \arrow[r,"\mathrm{Coeff}_{\mathcal{S}, k-2} \circ \mathrm{Res}_{k-2}"' swap] \arrow[d, "\widetilde{\vartheta}_k"'] & \C^{|\mathcal{S}|} \arrow[r] & 0 \\
	& & \mathcal{E}_k(\Gamma_1(p_1p_2))_0^{ (p_1, p_2)}\arrow[ru,"\mathcal{L}_{\mathcal{S},k}"'] & 
	\end{tikzcd}
	\]
	commutates, as the above sequence is exact and $\widetilde{\vartheta}_k$ is an isomorphism. By definition it is clear that $\xi_{p_1,p_2,k} = \widetilde{\vartheta}_k \circ \Xi_{p_1,p_2,k}$. So we are left to show that $\mathcal{L}_{\mathcal{S},k} = \mathrm{Coeff}_{\mathcal{S}, k-2} \circ \mathrm{Res}_{k-2} \circ \widetilde{\vartheta}_k^{-1}$. To show this w use Eichler duality, and it suffices to do it for basis vectors. We find with \eqref{eq:dual-alpha}
	\begin{align}
	\nonumber & \mathcal{L}_{\mathcal{S},k}\left( \vartheta_k\left(\widetilde{\alpha}_c^{(p_1)} \otimes \widehat{\alpha}_d^{(p_2)} ; p_2\tau\right) \right) \\
	\nonumber  = \ & \frac{(-2\pi i)^k}{(k-2)!} \frac{p_2^{k-1}}{4\pi^2} \\
	\nonumber  & \times \mathrm{Coeff}_{\mathcal{S}, k-2}\left( \sum_{\ell=0}^{k-2} \binom{k-2}{\ell} i^{1-\ell} (2\pi)^{-\ell-1} \Gamma(\ell+1) L\left( \vartheta_k\left(\widetilde{\alpha}_c^{(p_1)} \otimes \widehat{\alpha}_d^{(p_2)} ; p_2\tau\right); \ell+1\right) (p_1 \tau)^\ell \right) \\
	\nonumber  = \ & \frac{(-2\pi i)^k}{(k-2)!} \frac{p_2^{k-1}}{4\pi^2} \mathrm{Coeff}_{\mathcal{S}, k-2}\left(\sum_{\chi, \psi} \chi(-1)a_{\chi}(c) a_\psi(d) \sum_{\ell=0}^{k-2} \binom{k-2}{\ell} i^{1-\ell} (2\pi)^{-\ell-1} \Gamma(\ell+1) \right.\\
	\nonumber  & \hspace{3cm} \left. \times \cG(\chi)\cG(\psi) L\left( \vartheta_k\left(\omega_{\overline{\chi}} \otimes \omega_{\overline{\psi}} ; p_2\tau\right); \ell+1\right) (p_1 \tau)^\ell\right)
	\intertext{and with $M = p_1$ and $N = p_2$ in \eqref{eq:weakeisenstein}}
	\nonumber  = \ & \frac{(-2\pi i)^k}{(k-2)!} \frac{p_2^{k-1}}{4\pi^2} \mathrm{Coeff}_{\mathcal{S}, k-2}\left(\sum_{\chi, \psi} \chi(-1)a_{\chi}(c) a_\psi(d) \sum_{\ell=0}^{k-2} \binom{k-2}{\ell} i^{1-\ell} (2\pi)^{-\ell-1} \Gamma(\ell+1) \right.\\
	\nonumber  & \hspace{3cm} \times \left.\cG(\chi)\cG(\psi) \frac{p_2 (k-1)! \cG(\overline{\chi})}{\chi(-1) (-2\pi i)^{k} \cG(\psi)} L\left(E_k(\chi, \psi; p_2 \tau); \ell+1\right) (p_1 \tau)^{\ell}\right) \\
	\nonumber  = \ & \frac{(-2\pi i)^k}{(k-2)!} \frac{p_2^{k-1}}{4\pi^2} \mathrm{Coeff}_{\mathcal{S}, k-2}\left(\sum_{\chi, \psi} \chi(-1)a_{\chi}(c) a_\psi(d) \sum_{\ell=0}^{k-2} \binom{k-2}{\ell} i^{1-\ell} (2\pi)^{-\ell-1} \Gamma(\ell+1)\right.  \\
	\nonumber  & \hspace{3cm} \left.\times \frac{p_1p_2(k-1)!}{(-2\pi i)^k}   L\left(E_k(\chi, \psi; p_2 \tau); \ell+1\right) (p_1 \tau)^{\ell}\right) \\
	\nonumber  = \ & \frac{(-2\pi i)^k}{(k-2)!} \frac{p_2^{k-1}}{4\pi^2} \mathrm{Coeff}_{\mathcal{S}, k-2}\left(\sum_{\chi, \psi} \chi(-1)a_{\chi}(c) a_\psi(d) \frac{p_2(k-1)!}{(-2\pi i)^k} \sum_{\ell=0}^{k-2} \binom{k-2}{\ell} i^{1-\ell} \right. \\
	\nonumber   & \hspace{3cm} \left. \times \Lambda(E_k(\chi, \psi; p_2 \tau); \ell+1) p_1^{\ell+1} \tau^\ell\right). \\
	\intertext{With $N_\chi = p_1$ and $N_\psi = p_2$ in \eqref{eq:PeriodPolyomialIdentity} this equals to}
	\nonumber & \frac{(-2\pi i)^k}{(k-2)!} \frac{p_2^{k-1}}{4\pi^2} \mathrm{Coeff}_{\mathcal{S}, k-2}\left(\sum_{\chi, \psi} \chi(-1)a_{\chi}(c) a_\psi(d) \frac{p_2(k-1)!}{(-2\pi i)^k} \frac{4\pi^2 \chi(-1)}{p_2^{k} (k-1)}  \right. \\
	\nonumber & \hspace{3cm} \left. \times \res_{z=0}\left( z^{1-k} \omega_\psi(z) \omega_{\chi} \left( z\tau\right)\right) \right)\\
	\nonumber = \ & \frac{(-2\pi i)^k}{(k-2)!} \frac{p_2^{k-1}}{4\pi^2} \mathrm{Coeff}_{\mathcal{S}, k-2}\left(\frac{(k-2)!}{(-2\pi i)^k} \frac{4\pi^2}{p_2^{k-1}}  \sum_{\chi, \psi} a_{\chi}(c) a_\psi(d)  \res_{z=0}\left( z^{1-k} \omega_\psi(z) \omega_{\chi} \left( z\tau\right)\right)\right) \\
	\intertext{and with \eqref{eq:original-alpha}}
	\nonumber = \ & \mathrm{Coeff}_{\mathcal{S}, k-2}\left(\res_{z=0}\left( z^{1-k} \alpha^{(p_2)}_d(z) \alpha_{c}^{(p_1)} \left( z\tau\right)\right)\right) = \mathrm{Coeff}_{\mathcal{S}, k-2} \left( \mathrm{Res}_{k-2} \left( \alpha_{d}^{(p_2)}(z)\alpha_c^{(p_1)}(z\tau)\right)\right) \\
	\nonumber = \ & \mathrm{Coeff}_{\mathcal{S}, k-2} \left( \mathrm{Res}_{k-2} \left( \widetilde{\vartheta}_k^{-1}\left( \alpha_c^{(p_1)} \otimes \alpha_{d}^{(p_2)}\right)\right)\right).
	\end{align}
	This proves the theorem. 
	\end{proof}
\begin{bem} \label{bem:commutative} Note that the diagram in the beginning of the proof of Theorem \ref{thm:mainsmallweight} stays commutative, as long as we {\it define} $\xi_{p_1,p_2,k}$ to be the map $\widetilde{\vartheta}_k \circ \Xi_{p_1,p_2,k-2} \colon V \to \mathcal{E}_k(\Gamma_1(p_1p_2))_0^{(p_1,p_2)}$ for any subspace $V \subset W_{p_1}^0 \otimes W_{p_2}^0$. 
	\end{bem}

For weights $k \geq 3$, subsets $\mathcal{S} \subset \{0,\ldots,k-2\}$ and congruence subgroups $\Gamma$, we consider the spaces $M^{\mathcal{S}}_k(\Gamma)$ defined by
\begin{align*}
M^{\mathcal{S}}_k(\Gamma) := \left\{ f \in M_k(\Gamma) \colon L(f; \ell+1) = 0 \text{ for all } \ell \in \mathcal{S}\right\}.
\end{align*}

A realistic looking dimension formula is a direct consequence of the upper theorem. 

\begin{cor} \label{cor:dimension} Let the same conditions hold as in Theorem \ref{thm:mainsmallweight}. Then we have
	\begin{align*}
	\dim \mathcal{E}_k^{\mathcal{S}}(\Gamma_1(p_1p_2))_0^{(p_1,p_2)} = \dim \mathcal{E}_k(\Gamma_1(p_1p_2))_0^{(p_1,p_2)} - |\mathcal{S}|.
	\end{align*}
	\end{cor}

For large weights (compared to $p_1$ and $p_2$), the situation is more subtle, as we have less control over the weak functions involved. However, Corollary\ref{cor:exactsequence-k} at least gives estimates. In this spirit we finally mention a relationship between the order of functions in $J_{p_1,p_2}^0$ already introduced and the vanishing of critical $L$-values. We introduce the space
\begin{align*}
V_{M,N,k}^{(\ell_1,\ell_2)} := V_{M,N}^{(\ell_1,\ell_2)} \cap (W_M^0 \otimes W_N^0)_k.
\end{align*}

While Theorem \ref{thm:mainsmallweight} worked well for small weights, Theorem \ref{thm:L-strips} can be applied particularly well for large weights. 

\begin{thm} \label{thm:L-strips} Let $p_1$ and $p_2$ be two odd prime numbers and $k \geq 3$ be an integer.  Let $\ell_1$ and $\ell_2$ be integers such that $\max\{0, p_2-k-1\} \leq \ell_1 \leq p_1-2$, $\max\{0, p_1-k-1\} \leq \ell_2 \leq p_2-2$, and $\ell_1+\ell_2 \leq k-1$. We assign $\ell_1$ and $\ell_2$ a space 
\begin{align*}
\mathcal{E}_k^{(\ell_1, \ell_2)}(\Gamma_1(p_1p_2))_0^{(p_1,p_2)} := \left< f \in \mathcal{E}_k(\Gamma_1(p_1p_2))_0^{(p_1,p_2)}  \Big{|} \ \ord\left(\widetilde{\vartheta}_k^{-1}(f)\right) \geq (\ell_1, \ell_2) \right>,
\end{align*}
where $\widetilde{\vartheta}_k$ is the isomorphism defined in \eqref{eq:tildetheta}. Then we have
\begin{align*}
\mathcal{E}_k^{(\ell_1, \ell_2)}(\Gamma_1(p_1p_2))_0^{(p_1,p_2)} \subset \mathcal{E}_k^{\{0,\ldots,\ell_1-1\} \cup \{k-1-\ell_2,\ldots,k-2\}}(\Gamma_1(p_1p_2))_0^{(p_1,p_2)}.
\end{align*}
\end{thm}
\begin{proof} The map $\Xi_{p_1,p_2,k-2}$ is order preserving by Theorem \ref{thm:iso}. So $\ord(\widetilde{\vartheta}_k^{-1}(f)) \geq (\ell_1, \ell_2)$ is equivalent to $\ord(\Xi_{p_1,p_2,k}^{-1}(\widetilde{\vartheta}_k^{-1}(f)))\geq (\ell_1, \ell_2)$. Hence by Theorem \ref{thm:iso} we find
	\begin{align*}
	\Xi_{p_1,p_2,k}^{-1}(\widetilde{\vartheta}_k^{-1}(f)) & \in V_{p_1,p_2,k}^{(\ell_1,\ell_2)} = \bigoplus_{\substack{\ell_1 \leq c \\ \ell_2 \leq d \\ c+d \equiv k-2 \pmod{2}}} \C\left( \alpha_c^{(M)} \otimes \alpha_d^{(N)}\right) \\ 
	& \subset \bigoplus_{\substack{\left( m \notin \mathcal{S} \wedge k-2-n \notin \mathcal{S}\right) \\ \vee \left( (m,k-2-n) \in \mathcal{S}^2 \wedge  m \not= k-2-n \right) \\ m+n \equiv k-2 \pmod{2}}} \C \left( \alpha_m^{(M)} \otimes \alpha_n^{(N)}\right), 
	\end{align*}
	where $\mathcal{S} := \{0, \ldots, \ell_1-1\} \cup \{k-1-\ell_2, \ldots, k-2\}$, as $c \geq \ell_1$ and $c \in \mathcal{S}$ implies $p_1 -2 \leq k-1+\ell_2 \leq c \leq p_1-3$, which is absurd, and similarly $d \geq \ell_2$ and $k-2-d \in \mathcal{S}$ implies $p_2 - 2 \leq k-1-\ell_1 \leq d \leq p_2-3$, a contradiction. The claim now follows with Corollary \ref{cor:exactsequence-k} and Remark \ref{bem:commutative}.
	\end{proof}

\begin{bsp} \label{bsp:55-bsp} With Example \ref{bsp:Basiswechsel} together with Theorem \ref{thm:L-functions} and \eqref{eq:weakeisenstein}, we find that for any even $k \geq 4$ the non-trivial modular form
	\begin{align*}
	f(\tau) = C E_k(\overline{\chi_5}, \overline{\chi_5}; 5\tau) + 5 \left( E_k(\overline{\chi_5},\chi_5; 5\tau) +  E_k(\chi_5, \overline{\chi_5}; 5\tau)\right) + \overline{C} E_k(\chi_5, \chi_5; 5\tau),
	\end{align*}
	where after normalization $C$ can be chosen as 
	\begin{align*}
	C := - i (-3-4 i)^{\frac{3}{4}} (-3+4 i)^{\frac14},
	\end{align*}
	satisfies
	\begin{align*}
	L(f;1) = L(f;2) = L(f;k-2) = L(f;k-1) = 0.
	\end{align*}
	Recall that in Example \ref{bsp:Basis} we found that $\ord(\alpha_2^{(5)}(z)\alpha_2^{(5)}(\tau z)) = (2,2)$. Note that in this example only the vanishing at $1$ and $k-1$ is non-trivial, as $\chi_5$ is an odd character. 
\end{bsp}

\renewcommand\refname{References}

\end{document}